\documentclass[11pt,twoside]{amsart}
\usepackage{amsmath, amsthm, amscd, amsfonts, amssymb, graphicx}
\usepackage[bookmarksnumbered, plainpages]{hyperref}
 
\usepackage{mathrsfs}
\UseRawInputEncoding

\newtheorem{thm}{Theorem}[section]
\newtheorem{cor}[thm]{Corollary}
\newtheorem{lem}[thm]{Lemma}
\newtheorem{prop}[thm]{Proposition}
\newtheorem{defn}[thm]{Definition}
\newtheorem{rem}[thm]{\bf{Remark}}

\numberwithin{equation}{section}

\begin{document}

%------------------------------------------------------------------------------------%

\pagestyle{myheadings}

\bigskip
\bigskip
%------------------------------------------------------------------------------------%
%------------------------------------------------------------------------------------%

\title{Morse index and nullity of rotationally symmetric \texorpdfstring{$Y$}{Y}-singular minimal surfaces}

\author{Elham Matinpour}
\address{Department of Mathematics, Johns Hopkins University, 3400 N. Charles Street, Baltimore, MD 21218}
\email{ematinp1@jhu.edu}


\begin{abstract}
We compute the Morse index and nullity of a one-parameter family of rotationally symmetric minimal $Y$-singular surfaces in $\mathbb R^3$ whose singular set is a circle. The second variation is studied on a weighted energy space adapted to the noncompact faces and decomposed into angular Fourier modes. The resulting modewise problem separates the fixed-junction Dirichlet contribution from a finite-dimensional junction form. The exceptional resonant radial mode is treated directly. We prove that the $Y$-catenoid has index one and nullity three, while $Y_\alpha$ has index two and nullity five for $0<\alpha\leq\pi/3$.
\end{abstract}

\maketitle

\section{Introduction}
Minimal Plateau surfaces arise naturally in soap-film experiments and may contain singularities of $Y$- or $T$-type. See Plateau \cite{plateau1873statique} and Bernstein--Maggi \cite{bernstein2021symmetry}. In this paper we study a rotationally symmetric family of minimal $Y$-surfaces in $\mathbb R^3$ whose singular set is a single circle. We use the triple-junction formulation, in which three minimal faces meet at $120^\circ$ along their common boundary.

The family is denoted by $Y_\alpha$, $\alpha\in(0,\pi]$, where $\alpha$ is the contact angle of a distinguished face with the horizontal plane. See Definition \ref{def: Y-noid}. Up to reflection through that plane, it is enough to consider $\alpha\in(0,\pi/3]$ (Remark \ref{rem: Y-alphas}). As $\alpha\to0^+$, the family degenerates to the union of a $Y$-catenoid and the horizontal plane through the junction. We define $Y_0$ to be the $Y$-catenoid with that horizontal plane removed (Definition \ref{def: y-catenoid}). The endpoint $Y_{\pi/3}$ is the pseudo $Y$-catenoid (Definition \ref{def: pseudo y-catenoid}).

Our main result is the following.
\begin{thm} \label{thm: main result}
Let $Y_\alpha$, $\alpha \in [0,\tfrac{\pi}{3}]$, be as above.
\begin{enumerate}
    \item[(a)] The $Y$-catenoid $Y_0$ has Morse index one and nullity three.
    \item[(b)] For $0< \alpha < \tfrac{\pi}{3}$, the $Y$-noid $Y_\alpha$ has Morse index two and nullity five.
    \item[(c)] The pseudo $Y$-catenoid $Y_{\pi/3}$ has Morse index two and nullity five.
\end{enumerate}
\end{thm}
Here nullity means the weak nullity on the noncompact variational space $\mathcal D(Y_\alpha)$:
\[
\dim\{u\in\mathcal D(Y_\alpha):Q(u,v)=0\text{ for every }v\in\mathcal D(Y_\alpha)\}.
\]
The $Y$-catenoid is distinguished from the rest of the family by having index one rather than two.

We use Wang's triple-junction framework \cite{wang2022curvature}. The index calculation separates the fixed-junction problem from the boundary contribution, following the Dirichlet-to-Neumann viewpoint used by Tran \cite{tran2016index}. Rotational symmetry then reduces the boundary analysis to angular Fourier modes. Because the faces are noncompact, the analysis is carried out on a weighted energy space. The exceptional radial sector of $Y_{\pi/6}$ is treated directly rather than through the compact extension construction. Section \ref{sec: 2} fixes the notation. Section \ref{sec: index evaluation on catenoidal} develops the variational framework and index decomposition. Section \ref{section 4} computes the index and nullity of the family.

\subsection*{Acknowledgments}
I thank my advisor, Jacob Bernstein, for suggesting this problem and for his advice, encouragement, and many helpful comments. I also thank the anonymous referee for a careful reading of the manuscript and for valuable comments and suggestions that substantially improved the paper.

%%%%%%%%%%%%%%%%%%%%%%%%%%%%%%%%%%%%%%%%%%%%%%%%%%%%%%%%%%%%%%%%%%%%%%%%%%%%%%%%%%%%%%%%%%%%%%%%%%%%%%%%%%%%%%%%%%
\section{Background and Notation} \label{sec: 2}

\subsection{Basic Function Spaces} \label{subsec: basic func. spaces}
Let $M_1,M_2,M_3$ be two-dimensional manifolds with boundary. For each $i$, let $\Gamma_i\subset\partial M_i$ and fix a diffeomorphism $\phi_i:\Gamma_i\to\Gamma$, where $\Gamma$ is a one-dimensional manifold. We identify the three copies of $\Gamma$ and write
\[
M=\left(M_1\coprod M_2\coprod M_3\right)/\!\sim,
\qquad
\partial' M_i=\partial M_i\setminus\Gamma_i,
\]
where $x\in\Gamma_i$ and $y\in\Gamma_j$ are equivalent when $\phi_i(x)=\phi_j(y)$. A facewise function on $M$ will be written $u=(u_1,u_2,u_3)$, and its trace from the $i$-th face is denoted by $u|_\Gamma^i$.

We use the following spaces:
\[
\begin{aligned}
C_{\neq}^k(M)&=\bigoplus_{i=1}^3 C^k(M_i),\\
C_=^k(M)&=\{u\in C_{\neq}^k(M):u|_\Gamma^1=u|_\Gamma^2=u|_\Gamma^3\},\\
C_\Delta^k(M)&=\{u\in C_{\neq}^k(M):u|_\Gamma^1+u|_\Gamma^2+u|_\Gamma^3=0\},\\
C_v^k(M)&=\{u\in C_{\neq}^k(M):u|_\Gamma^i=0,\,i=1,2,3\}.
\end{aligned}
\]
The same notation is used componentwise for maps into $\mathbb R^n$. A subscript $0$ means that the function also vanishes on the exterior boundary $\coprod_i\partial' M_i$.

\subsection{Triple Junction Minimal Surfaces} \label{subsec: triple junction surface}
Let
\[
\mathbf F\in C_=^\infty(M;\mathbb R^3)
\]
be such that each $\mathbf F|_i:M_i\to\mathbb R^3$ is an immersion. Choose orientations on the faces and compatible unit normals $\mathbf n_i$. We call $\mathbf F$ a \emph{triple-junction minimal surface} if each face is minimal and
\[
\mathbf n=(\mathbf n_1,\mathbf n_2,\mathbf n_3)\in C_\Delta^\infty(M;\mathbb R^3).
\]
Thus the three faces meet at $120^\circ$ along $\mathbf F(\Gamma)$.

If $\mathbf F_t\in C_=^\infty(M;\mathbb R^3)$ is a variation of $\mathbf F=\mathbf F_0$ with variation field $\mathbf V$, then its normal component
\[
\phi=\mathbf V\cdot\mathbf n
\]
belongs to $C_\Delta^\infty(M)$. Wang's second-variation formula \cite{wang2022curvature} is
\begin{equation} \label{eqn: 2nd variation fromula}
Q(\phi)=\sum_{i=1}^3\left[
\int_{M_i}\left(|\nabla_{M_i}(\phi|_i)|^2-|A_{\mathbf F(M_i)}|^2(\phi|_i)^2\right)\,dA_i
-\int_\Gamma (\mathbf H_{\mathbf F(\Gamma)}\cdot\nu_i)(\phi|_\Gamma^i)^2\,ds
\right],
\end{equation}
where $\nu_i$ is the outward unit conormal to the $i$-th face along the junction. We use the same symbol $Q$ for the associated symmetric bilinear form.

Let $W_{\Delta,c}^{1,2}(M)$ denote the compactly supported facewise $W^{1,2}$ functions $u=(u_i)$ that vanish on the exterior boundary and satisfy the junction condition $\sum_i u_i|_\Gamma=0$. The Morse index is the maximal dimension of a subspace of $W_{\Delta,c}^{1,2}(M)$ on which $Q$ is negative definite.

\begin{rem}
\rm
For the noncompact $Y$-noids it is convenient to work on the larger space $\mathcal D(M)$ introduced in Subsection \ref{subsec: index decomp}. Proposition \ref{prop: function space for noncompact index} shows that this gives the same Morse index.
\end{rem}

\subsection{Family of \texorpdfstring{$Y$}{Y}-Noids} \label{subsec: family of y-noids}
\begin{defn} \label{def: Y-noid}
\rm
Assume that $\Gamma$ is a circle centered at the origin in the horizontal plane. Let $\Sigma_1,\Sigma_2,\Sigma_3$ be rotational minimal faces meeting along $\Gamma$, with outward unit conormals $\eta_i$ satisfying
\[
\eta_1+\eta_2+\eta_3=0.
\]
Let $\alpha_i\in(0,\pi]$ be the contact angle of $\eta_i$ with the horizontal plane and set $\alpha=\alpha_1$. The resulting minimal triple-junction surface
\[
Y_\alpha=(\Sigma_1,\Sigma_2,\Sigma_3;\Gamma)
\]
is called a $Y$-noid. Figure \ref{fig: alpha} illustrates the contact-angle convention in a meridional cross-section.
\end{defn}

\begin{figure}
\centering
\includegraphics[width=0.88\textwidth]{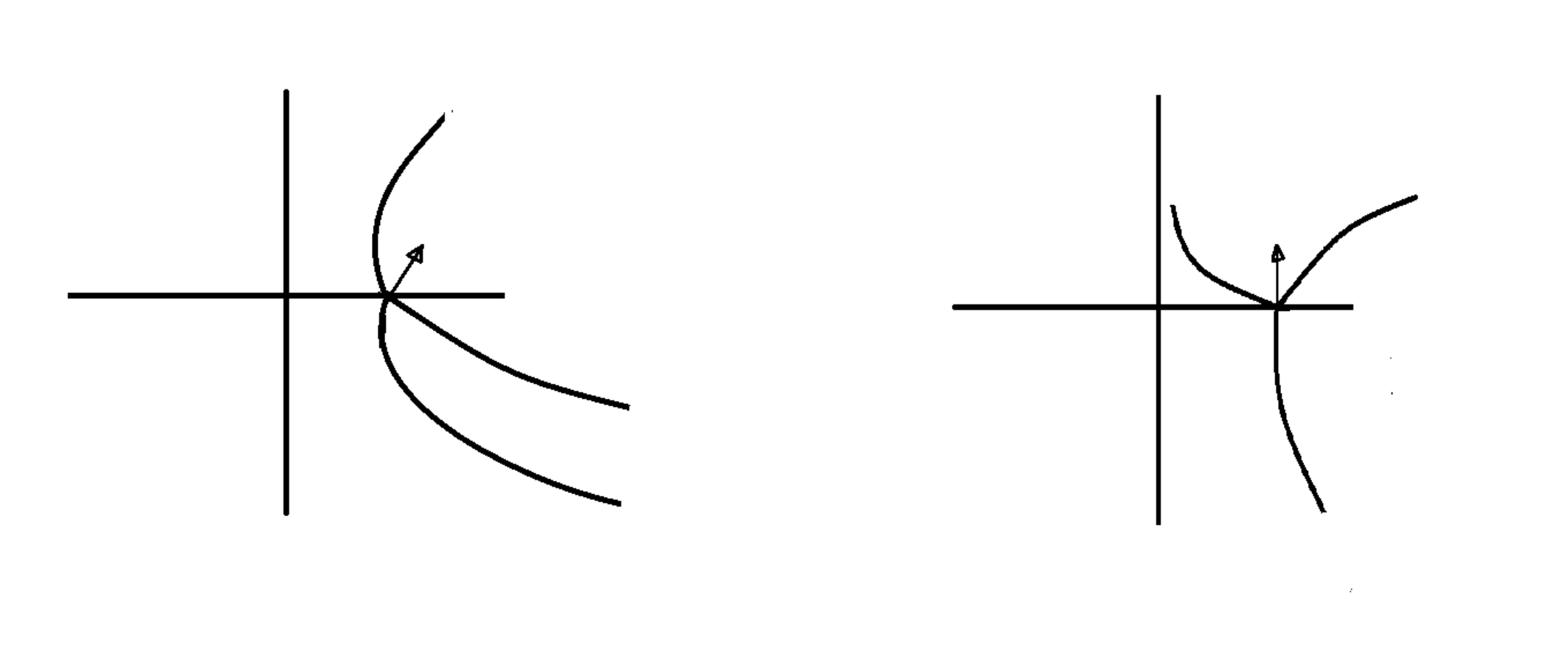}
\caption{Schematic meridional cross-sections illustrating the contact-angle convention for the family $Y_\alpha$.}
\label{fig: alpha}
\end{figure}

\begin{rem} \label{rem: Y-alphas}
\rm
By the classical classification of cyclic minimal surfaces \cite{riemann1867flache,enneper1869cyklischen,nitsche1989lectures}, together with \cite{pyo2012remarks}, each face occurring here is a piece of a catenoid or of the horizontal plane. Rotating the conormal $Y$ along $\Gamma$ produces the one-parameter family $\{Y_\alpha\}_{\alpha\in(0,\pi]}$. Up to reflection through the horizontal plane, every member is represented by some $\alpha\in(0,\pi/3]$. In particular, $Y_\alpha\cong Y_{2\pi-\alpha}$.
\end{rem}

\begin{rem}
\rm
For a single face with contact angle $\alpha$, the limit as $\alpha\to0$ is the disk $D_\Gamma$ together with the horizontal plane. At $\alpha=\pi$ the face is the exterior planar region $\{x_3=0\}\setminus D_\Gamma$. For $0<\alpha<\pi$ it is catenoidal. The surfaces $Y_\pi$ and $Y_{\pi/3}$ have the same geometry.
\end{rem}

\begin{rem}
\rm
For $0<\alpha<\pi/3$, all three faces of $Y_\alpha$ are non-planar catenoidal faces. At $\alpha=\pi/3$, one face is planar and the other two are catenoidal.
\end{rem}

As $\alpha\to0$, the limiting configuration contains the horizontal plane in addition to the triple-junction part. We remove that plane and use the following convention.

\begin{defn} \label{def: y-catenoid}
\rm
The \textbf{$Y$-catenoid} $Y_0=(\Sigma_1,\Sigma_2,\Sigma_3;\Gamma)$ consists of the disk $\Sigma_1=D_\Gamma$ and two symmetric catenoidal faces $\Sigma_2,\Sigma_3$ meeting it at $120^\circ$ along $\Gamma$.
\end{defn}

\begin{defn} \label{def: pseudo y-catenoid}
\rm
The \textbf{pseudo $Y$-catenoid} $Y_{\pi/3}$ has one exterior planar face
\[
\Sigma_2=D_\Gamma^c=\{x_3=0\}\setminus D_\Gamma
\]
and two symmetric catenoidal faces $\Sigma_1,\Sigma_3$ meeting it at $60^\circ$ along $\Gamma$.
\end{defn}

\section{Variational Framework and Index Decomposition} \label{sec: index evaluation on catenoidal}
We now set up the variational framework used to compute the index of the $Y$-noids. The main points are the coupled boundary condition at the junction and the separation into angular modes.

\subsection{Index Decomposition} \label{subsec: index decomp}
On each $M_i$, let
\[
J_i=\Delta_{M_i}+|A_{\mathbf F(M_i)}|^2,
\qquad
\beta_i=\mathbf H_{\mathbf F(\Gamma)}\cdot\nu_i,
\]
and identify traces on $\Gamma_i$ with functions on $\Gamma$. With all conormals outward, the bilinear index form is
\[
Q(u,v)=\sum_{i=1}^3\int_{M_i}
\bigl(\langle\nabla u_i,\nabla v_i\rangle-|A_{\mathbf F(M_i)}|^2u_iv_i\bigr)\,dA_i
-\sum_{i=1}^3\int_\Gamma\beta_i u_iv_i\,ds,
\]
with the trace condition $\sum_i u_i|_\Gamma=0$.

\subsubsection*{The noncompact variational space}
Following Chodosh and Maximo \cite[Section 3]{chodosh2018topology}, set
\[
\omega(x)=\frac{1}{(1+|x|^2)\log^2(2+|x|)}.
\]
Let $L_*^2(M)$ be the weighted $L^2$ space of measurable triples $u=(u_i)$ with
\[
\|u\|_{L_*^2(M)}^2=\sum_{i=1}^3\int_{M_i}\omega u_i^2\,dA_i<\infty.
\]
We define
\[
\mathcal D(M)=\left\{u=(u_i):
\begin{array}{l}
u_i\in H^1_{\rm loc}(M_i)\text{ up to }\partial M_i,\quad
\sum_i\int_{M_i}|\nabla u_i|^2\,dA_i<\infty,\\[2pt]
u\in L_*^2(M),\quad \sum_i u_i|_\Gamma=0
\end{array}\right\}.
\]
Here $H^1_{\rm loc}(M_i)$ up to $\partial M_i$ means ordinary $H^1$ regularity on every relatively compact subset, including boundary collars. We equip this space with the norm
\[
\|u\|_{\mathcal D(M)}^2
=\sum_i\int_{M_i}|\nabla u_i|^2\,dA_i+\|u\|_{L_*^2(M)}^2.
\]
The two terms in the norm play different roles. The Dirichlet energy is unweighted because it is the energy appearing in the index form, while the weight is used only for the zeroth-order term. This allows bounded Jacobi fields such as $\tanh(t+T)$ to belong to $\mathcal D(M)$ even when they are not in ordinary $L^2$.

For one face, $\mathcal D(M_i)$ denotes the corresponding space without the junction condition. In particular, no unweighted $L^2$ condition is imposed on an unbounded face. Proposition \ref{prop: function space for noncompact index} shows that the index on $\mathcal D(M)$ agrees with the compactly supported Morse index of Section \ref{subsec: triple junction surface}. We define
\[
\ker_{\mathcal D}Q=\{u\in\mathcal D(M):Q(u,v)=0\text{ for every }v\in\mathcal D(M)\},
\qquad
\operatorname{nul}_{\mathcal D}(M)=\dim\ker_{\mathcal D}Q.
\]
Thus facewise Jacobi fields or vectors satisfying only $Q(u,u)=0$ need not lie in the null space.

\subsubsection*{The operator associated with the full form}
For the compact comparison problem, let the $M_i$ be compact smooth faces with common boundary $\Gamma$ and exterior boundary $\partial'M_i$. We use $H^{1/2}(\Gamma)$ for the Sobolev trace space of $H^1$ functions and $H^{-1/2}(\Gamma)$ for its dual. Let
\[
J_{i,V}=\Delta_{M_i}+V_i,
\qquad
\mathcal H(M)=\left\{u\in\bigoplus_iH^1(M_i):
u_i|_{\partial'M_i}=0,\ \sum_i u_i|_\Gamma=0\right\},
\]
and
\[
q_{V,b}(u,v)=\sum_i\int_{M_i}(\langle\nabla u_i,\nabla v_i\rangle-V_iu_iv_i)\,dA_i
+\sum_i\int_\Gamma b_i u_iv_i\,ds.
\]
The geometric form is obtained from $V_i=|A_i|^2$ and $b_i=-\beta_i$. By the trace inequality, $q_{V,b}$ is densely defined, closed, and lower semibounded on $\mathscr L^2(M)=\bigoplus_iL^2(M_i)$. Hence it determines a self-adjoint operator $\mathscr A_{V,b}$ characterized by $q_{V,b}(u,v)=(\mathscr A_{V,b}u,v)_{\mathscr L^2(M)}$ for $u$ in its operator domain and every $v$ in the form domain.

If $u_i\in H^1(M_i)$ and $J_{i,V}u_i\in L^2(M_i)$, its weak outward conormal derivative is defined by
\[
\langle\partial_{\nu_i}u_i,a\rangle_\Gamma
=\int_{M_i}\bigl(\langle\nabla u_i,\nabla w_i\rangle+(\Delta_{M_i}u_i)w_i\bigr)\,dA_i,
\]
where $w_i$ is any $H^1$ lifting of $a\in H^{1/2}(\Gamma)$ that vanishes on $\partial'M_i$. This lies in $H^{-1/2}(\Gamma)$ and agrees with the classical outward conormal derivative for smooth $u_i$. Green's formula gives
\[
q_{V,b}(u,v)=\sum_i\int_{M_i}(-J_{i,V}u_i)v_i\,dA_i
+\sum_i\langle\partial_{\nu_i}u_i+b_i u_i|_\Gamma,v_i|_\Gamma\rangle_\Gamma.
\]
Testing with compatible traces $(a,-a,0)$ and $(a,0,-a)$ shows that
\[
\operatorname{Dom}(\mathscr A_{V,b})
=\left\{u\in\mathcal H(M):J_{i,V}u_i\in L^2(M_i),\ 
\partial_{\nu_i}u_i+b_i u_i|_\Gamma=r\ (i=1,2,3)\right\},
\]
for one $r\in H^{-1/2}(\Gamma)$, not necessarily constant along $\Gamma$.

A volume eigenfunction therefore satisfies
\[
-J_{i,V}u_i=\lambda u_i,
\qquad \sum_i u_i|_\Gamma=0,
\qquad \partial_{\nu_i}u_i+b_i u_i|_\Gamma=r.
\]
The coupled Steklov problem
\[
q_{V,b}(u,v)=\sigma\sum_i\int_\Gamma u_i v_i\,ds
\qquad(v\in\mathcal H(M))
\]
has instead
\[
J_{i,V}u_i=0,
\qquad \sum_i u_i|_\Gamma=0,
\qquad \partial_{\nu_i}u_i+b_i u_i|_\Gamma-\sigma u_i|_\Gamma=r.
\]
Thus individual facewise Steklov conditions do not by themselves give the coupled boundary condition, and a facewise Jacobi function need not have a single Steklov mode as its trace. For the actual $Y$-noids, the complete mode decomposition is proved below directly on $\mathcal D(M)$.

\subsection{General Index Decomposition}
We first treat a compact comparison problem. Assume the compact hypotheses of Subsection \ref{subsec: index decomp} and use the form $q=q_{V,b}$. Put
\[
q_i(u,v)=\int_{M_i}(\langle\nabla u,\nabla v\rangle-V_iuv)\,dA_i,
\qquad
\mathcal K_i=\{k\in H_0^1(M_i):J_{i,V}k=0\}.
\]
Here $q_i$ is only the interior Dirichlet form associated with $J_{i,V}$. The junction term involving $b_i$ belongs to the full coupled form $q_{V,b}$. It vanishes on $H_0^1(M_i)$ and reappears when the three faces are coupled along $\Gamma$. The spaces $\mathcal K_i$ are finite-dimensional and smooth up to the boundary.

\begin{prop}\label{prop: compact normalized extension}
For each $f\in H^{1/2}(\Gamma_i)$ there is a unique $\hat E_i(f)\in H^1(M_i)$ such that
\[
\hat E_i(f)|_{\Gamma_i}=f,
\qquad \hat E_i(f)|_{\partial'M_i}=0,
\qquad J_{i,V}\hat E_i(f)\in\mathcal K_i,
\qquad \hat E_i(f)\perp_{L^2(M_i)}\mathcal K_i.
\]
The map $\hat E_i:H^{1/2}(\Gamma_i)\to H^1(M_i)$ is bounded and linear, and it preserves smoothness.
\end{prop}
\begin{proof}
Choose a bounded trace lifting $\ell_i f$ and define $h_f\in\mathcal K_i$ by
\[
(h_f,k)_{L^2(M_i)}=-\int_{\Gamma_i}f\,\partial_{\nu_i}k\,ds
\qquad(k\in\mathcal K_i).
\]
Green's formula shows that the functional
\[
v\longmapsto-q_i(\ell_i f,v)-(h_f,v)_{L^2(M_i)}
\]
annihilates $\mathcal K_i$. The Fredholm alternative for the Dirichlet problem therefore gives a unique $w\in H_0^1(M_i)\cap\mathcal K_i^{\perp}$ solving the corresponding weak equation. If $P_i$ is the $L^2$ projection onto $\mathcal K_i$, then
\[
\hat E_i(f)=\ell_i f+w-P_i(\ell_i f+w)
\]
has the required properties. The same Fredholm estimate gives boundedness. If two normalized extensions have the same trace, their difference has zero trace, lies orthogonal to $\mathcal K_i$, and has Jacobi operator in $\mathcal K_i$. Pairing with $\mathcal K_i$ shows that this difference vanishes. Smoothness follows from elliptic regularity.
\end{proof}

For $f=(f_1,f_2,f_3)\in H^{1/2}(\Gamma;\mathbb R^3)$ define the simultaneous extension $E(f)$ by
\[
E(f)|_{M_i}=\widehat E_i(f_i),\qquad i=1,2,3,
\]
where the boundary identifications $\Gamma_i\simeq\Gamma$ are understood. Thus $E(g)$ and $E(\varphi_a)$ below refer to this same componentwise extension. Put
\[
\mathcal K=\bigoplus_i\mathcal K_i,
\qquad
\mathcal T_\Delta=\{f\in H^{1/2}(\Gamma;\mathbb R^3):f_1+f_2+f_3=0\},
\]
\[
P_\Delta f=f-\tfrac13(f_1+f_2+f_3)(1,1,1),
\qquad
\mathcal N=\{\partial_\nu k:k\in\mathcal K\}.
\]
Orthogonality of traces is taken in boundary $L^2$, with inner product $(f,g)_\Gamma=\sum_i\int_\Gamma f_i g_i\,ds$. For a subspace $S$, we write $S^{\perp_\Gamma}=\{f:(f,s)_\Gamma=0\text{ for every }s\in S\}$. Set
\[
\mathcal F=P_\Delta\mathcal N,
\qquad
\mathcal C=\mathcal T_\Delta\cap\mathcal N^{\perp_\Gamma}.
\]

\begin{lem}\label{lem: JE(f)=0}
For $f\in H^{1/2}(\Gamma;\mathbb R^3)$ and $k\in\mathcal K$,
\[
(J_VE(f),k)_{\mathscr L^2(M)}=-(f,\partial_\nu k)_\Gamma.
\]
$J_VE(f)=0$ exactly when $f\in\mathcal N^{\perp_\Gamma}$, and
\[
\mathcal T_\Delta=\mathcal F\oplus_{L^2(\Gamma)}\mathcal C.
\]
A compatible trace has a facewise Jacobi extension with zero exterior trace if and only if it belongs to $\mathcal C$. All such extensions differ by an element of $\mathcal K$.
\end{lem}
\begin{proof}
Green's formula gives
\[
-\int_{M_i}(J_{i,V}E(f)_i)k_i\,dA_i
=q_i(E(f)_i,k_i)=\int_\Gamma f_i\partial_{\nu_i}k_i\,ds,
\]
and summing proves the first claim. For compatible $g$, $(g,n)_\Gamma=(g,P_\Delta n)_\Gamma$, hence
$\mathcal C=\mathcal T_\Delta\cap\mathcal F^{\perp_\Gamma}$ and the stated orthogonal splitting follows by projection onto the finite-dimensional space $\mathcal F$. The solvability statement follows from the same Green identity, and uniqueness is modulo the Dirichlet kernel.
\end{proof}

The $L^2$ closure of $\mathcal C$ is
\[
\mathscr L_{\mathcal C}
=\overline{\mathcal C}^{\,L^2(\Gamma)}
=\mathscr L_\Delta\cap\mathcal N^{\perp_\Gamma},
\qquad
\mathscr L_\Delta=\{f\in L^2(\Gamma;\mathbb R^3):\sum_i f_i=0\},
\]
and its orthogonal projection is $P_{\mathcal C}=P_\Delta-P_{\mathcal F}$. These projections also act on $H^{-1/2}$ by duality.

Put
\[
\mathcal W=\left\{w\in\bigoplus_iH_0^1(M_i):w\perp_{\mathscr L^2(M)}\mathcal K\right\},
\qquad
\mathcal Z_{\mathcal F}=\mathcal K+E(\mathcal F).
\]
The following correction gives the full-flux normalization used in the decomposition. Compare \cite[Section 6.2.1, Proposition 6.13]{wang2022thesis}.

\begin{prop} \label{prop: decomp Kperp Z}
Let $m=\dim\mathcal F$. Choose an $L^2(\Gamma)$-orthonormal basis $\varphi_1,\ldots,\varphi_m$ of $\mathcal F$ and $k_a\in\mathcal K$ with $P_\Delta\partial_\nu k_a=\varphi_a$. For $g\in\mathcal C$, define
\[
\widetilde E(g)=E(g)-\sum_{a=1}^m q(E(\varphi_a),E(g))\,k_a.
\]
Then $J_V\widetilde E(g)=0$, $\widetilde E(g)|_\Gamma=g$, and
\[
\mathcal H(M)=\mathcal W\oplus_q\mathcal Z_{\mathcal F}\oplus_q\widetilde E(\mathcal C),
\]
a direct sum of mutually $q$-orthogonal closed subspaces. In addition,
\[
P_{\mathcal F}\bigl(\partial_\nu\widetilde E(g)+(b_i g_i)_{i=1}^3\bigr)=0.
\]
\end{prop}
\begin{proof}
\noindent Step 1. \emph{The correction and the mixed pairings.}
For $k,\ell\in\mathcal K$ and $f\in\mathcal T_\Delta$, Green's formula gives
\[
q(k,\ell)=0,
\qquad
q(k,E(f))=(P_\Delta\partial_\nu k,f)_\Gamma.
\]
Hence $q(k_a,E(\varphi_b))=\delta_{ab}$. If $g\in\mathcal C$, then $g\perp_\Gamma\mathcal N$, so $q(k,E(g))=0$ for every $k\in\mathcal K$. The definition of $\widetilde E$ therefore gives
\[
q(E(\varphi_b),\widetilde E(g))
=q(E(\varphi_b),E(g))
-\sum_a q(E(\varphi_a),E(g))\,q(E(\varphi_b),k_a)=0.
\]
For $w\in\mathcal W$, the zero trace of $w$ and the fact that $J_VE(f)\in\mathcal K$ give
\[
q(w,k)=0\quad(k\in\mathcal K),
\qquad
q(w,E(f))=0.
\]
These identities prove pairwise $q$-orthogonality of the three stated subspaces.

\medskip
\noindent Step 2. \emph{Exhaustiveness, directness, and closedness.}
Let $u\in\mathcal H(M)$ and let $f=u|_\Gamma\in\mathcal T_\Delta$. Write
\[
f=f_0+g,
\qquad f_0=P_\mathcal Ff\in\mathcal F,
\qquad g\in\mathcal C.
\]
Then
\[
v=u-E(f_0)-\widetilde E(g)
\]
has zero trace on every boundary component. Let $P_\mathcal K$ be the volume-$L^2$ projection onto $\mathcal K$, and put
\[
k=P_\mathcal K v,
\qquad w=v-k.
\]
Then $w\in\mathcal W$, $k+E(f_0)\in\mathcal Z_\mathcal F$, and
\[
u=w+\bigl(k+E(f_0)\bigr)+\widetilde E(g),
\]
which proves exhaustiveness.

To prove directness, suppose
\[
w+k+E(f_0)+\widetilde E(g)=0.
\]
Taking traces gives $f_0+g=0$. Since $\mathcal T_\Delta=\mathcal F\oplus_{L^2(\Gamma)}\mathcal C$, we have $f_0=g=0$. Thus $w+k=0$. The volume-$L^2$ orthogonality of $w$ and $\mathcal K$ gives $w=k=0$. The trace map, the finite-dimensional projection onto $\mathcal F$, the bounded maps $E$ and $\widetilde E$, and the volume projection onto $\mathcal K$ are all bounded. Hence the three decomposition projections are bounded in $H^1$, and their ranges are closed.

\medskip
\noindent Step 3. \emph{The full boundary flux.}
For $g\in\mathcal C$, the field $\widetilde E(g)$ is Jacobi on each face. Green's formula therefore gives, for every basis vector $\varphi_a$ of $\mathcal F$,
\[
0=q(E(\varphi_a),\widetilde E(g))
=\left\langle
\partial_\nu\widetilde E(g)+(b_i g_i)_{i=1}^3,\varphi_a
\right\rangle_\Gamma.
\]
The $\mathcal F$ component of the full boundary flux therefore vanishes, as required.
\end{proof}

Define
\[
\mathfrak s(g,h)=q(\widetilde E(g),\widetilde E(h)),
\qquad g,h\in\mathcal C.
\]

\begin{cor} \label{cor: Kperp Z index restriction}
In the compact setting,
\[
\operatorname{Ind}(q;\mathcal H(M))
=\sum_{i=1}^3Ind_{J_{i,V}}^0(M_i)
+\dim\mathcal F+\operatorname{Ind}(\mathfrak s;\mathcal C),
\]
and $\mathfrak s(g,h)=q(E(g),E(h))$ on $\mathcal C$.
\end{cor}
\begin{proof}
\noindent Step 1. \emph{The finite-dimensional block.}
Let $d=\dim\mathcal K$, $m=\dim\mathcal F$, and
\[
\mathcal K_0=\ker(P_\Delta\partial_\nu|_\mathcal K).
\]
The map $P_\Delta\partial_\nu:\mathcal K\to\mathcal F$ has rank $m$, so
\[
\mathcal K=\mathcal K_0\oplus\operatorname{span}\{k_1,\ldots,k_m\}.
\]
Put $e_a=E(\varphi_a)$ and $A_{ab}=q(e_a,e_b)$. In the ordered basis $k_1,\ldots,k_m,e_1,\ldots,e_m$, followed by a basis of $\mathcal K_0$, the form on $\mathcal Z_\mathcal F$ has matrix
\[
\begin{pmatrix}
0&I_m&0\\
I_m&A&0\\
0&0&0_{d-m}
\end{pmatrix}.
\]
The matrix $A$ is symmetric. Replacing
\[
e_a\quad\text{by}\quad e_a'=e_a-\frac12\sum_bA_{ab}k_b
\]
gives $q(e_a',e_b')=0$ and $q(k_a,e_b')=\delta_{ab}$. Hence the vectors $k_a-e_a'$ give $m$ mutually orthogonal negative directions and the vectors $k_a+e_a'$ give $m$ positive directions. The remaining space $\mathcal K_0$ lies in the bilinear kernel. Thus
\[
\operatorname{Ind}(q;\mathcal Z_\mathcal F)=m=\dim\mathcal F.
\]

\medskip
\noindent Step 2. \emph{The zero-trace and boundary blocks.}
The zero-trace space decomposes in volume $L^2$ as
\[
\bigoplus_iH_0^1(M_i)=\mathcal W\oplus\mathcal K.
\]
The form vanishes whenever one argument is in $\mathcal K$ and the other has zero trace. Projection onto $\mathcal W$ therefore preserves the form on negative subspaces. It follows that
\[
\operatorname{Ind}(q;\mathcal W)=\sum_{i=1}^3\operatorname{Ind}_{J_{i,V}}^0(M_i).
\]
For $g,h\in\mathcal C$, the differences $\widetilde E(g)-E(g)$ and $\widetilde E(h)-E(h)$ lie in $\mathcal K$. Their pairings with $E(g)$ and $E(h)$ vanish, as do pairings within $\mathcal K$. Hence
\[
\mathfrak s(g,h)=q(E(g),E(h)).
\]
The trace map is the inverse of $\widetilde E$ on its image, so the index on $\widetilde E(\mathcal C)$ is $\operatorname{Ind}(\mathfrak s;\mathcal C)$.

\medskip
\noindent Step 3. \emph{Additivity.}
Proposition \ref{prop: decomp Kperp Z} gives three mutually $q$-orthogonal summands. Negative subspaces chosen from these summands give the lower bound for the total index. Conversely, the projections of any finite-dimensional negative subspace lie in a finite-dimensional direct sum on which the form is block diagonal, so its negative dimension is at most the sum of the indices of the three full summands.
\end{proof}

\medskip
\noindent\emph{The compact coupled boundary operator.}
The form $\mathfrak s$ is densely defined on the boundary Hilbert space $\mathscr L_\mathcal C$ with form domain $\mathcal C$. We first record the estimate that makes the operator construction precise. For every $\varepsilon>0$ there is $C_\varepsilon$ such that
\[
\|E(g)\|_{L^2(M)}^2
\le
\varepsilon\|E(g)\|_{H^1(M)}^2
+C_\varepsilon\|g\|_{L^2(\Gamma)}^2,
\qquad g\in\mathcal C.
\]
Indeed, if this failed for some fixed $\varepsilon$, after normalization there would be $u_j=E(g_j)$ with
\[
\|u_j\|_{L^2(M)}=1,
\qquad
\|u_j\|_{H^1(M)}^2\le\varepsilon^{-1},
\qquad
\|g_j\|_{L^2(\Gamma)}\to0.
\]
Compactness of the faces gives a subsequence converging weakly in $H^1$ and strongly in $L^2$ to a limit $u$. The trace of $u$ is zero. Since $g_j\in\mathcal C$, Lemma \ref{lem: JE(f)=0} gives $J_Vu_j=0$, and the normalization in Proposition \ref{prop: compact normalized extension} gives $u_j\perp_{L^2(M)}\mathcal K$. Both properties pass to the limit. Thus $u\in\mathcal K\cap\mathcal K^{\perp_{L^2(M)}}=\{0\}$, contradicting $\|u\|_{L^2(M)}=1$.

Let $V_*=\max_i\|V_i\|_\infty$ and $b_*=\max_i\|b_i\|_\infty$. For $u=E(g)$, Corollary \ref{cor: Kperp Z index restriction} gives
\[
\mathfrak s(g,g)=q(u,u)
\ge
\|u\|_{H^1(M)}^2
-(1+V_*)\|u\|_{L^2(M)}^2
-b_*\|g\|_{L^2(\Gamma)}^2.
\]
Choose $\varepsilon$ so that $(1+V_*)\varepsilon\le1/2$. The preceding estimate and the trace theorem yield
\[
\mathfrak s(g,g)
\ge
c\|g\|_{H^{1/2}(\Gamma)}^2
-C\|g\|_{L^2(\Gamma)}^2.
\]
Boundedness of $E:H^{1/2}(\Gamma)\to H^1(M)$ and continuity of $q$ give the corresponding upper bound. Hence, for a sufficiently large fixed $\gamma>0$, there are constants $c_1,C_1>0$ such that
\[
c_1\|g\|_{H^{1/2}(\Gamma)}^2
\le
\mathfrak s(g,g)+\gamma\|g\|_{L^2(\Gamma)}^2
\le
C_1\|g\|_{H^{1/2}(\Gamma)}^2.
\]
The shifted form norm is therefore equivalent to the $H^{1/2}$ norm on $\mathcal C$, and $\mathfrak s$ is closed and lower semibounded on $\mathscr L_\mathcal C$.

Let $Bf=\partial_\nu E(f)$ and $\mathsf b f=(b_if_i)_i$. The self-adjoint boundary operator associated with $\mathfrak s$ is
\[
\operatorname{Dom}(T)=\{g\in\mathcal C:P_\mathcal C(B+\mathsf b)g\in\mathscr L_\mathcal C\},
\qquad
Tg=P_\mathcal C(B+\mathsf b)g.
\]
Here membership in $\mathscr L_\mathcal C$ means that the projected $H^{-1/2}$ distribution is represented by a boundary-$L^2$ function. Green's formula gives
\[
\mathfrak s(g,h)=\langle(B+\mathsf b)g,h\rangle_\Gamma
\qquad(g,h\in\mathcal C).
\]
Proposition \ref{prop: decomp Kperp Z} and the fact that $\widetilde E(g)-E(g)\in\mathcal K$ give
\[
P_\mathcal C(B+\mathsf b)g
=P_\Delta\bigl(\partial_\nu\widetilde E(g)+\mathsf b g\bigr)
\quad\text{in }H^{-1/2}(\Gamma;\mathbb R^3).
\]

The shifted resolvent maps $\mathscr L_\mathcal C$ boundedly into $\mathcal C$ with its $H^{1/2}$ norm. Since $H^{1/2}(\Gamma)\hookrightarrow L^2(\Gamma)$ is compact, the resolvent is compact. The compact spectral theorem gives an $L^2$-orthonormal basis $(g_j)$ of $\mathscr L_\mathcal C$ with $Tg_j=\sigma_jg_j$. The shifted form estimate shows that the expansion converges in $H^{1/2}$ for every $g\in\mathcal C$, and
\[
\operatorname{Ind}(\mathfrak s;\mathcal C)=\#\{j:\sigma_j<0\},
\]
with multiplicity. These are eigenfunctions of the coupled boundary operator, not a simultaneous scalar Steklov basis on the faces.

A weak Steklov solution with nonzero trace $g$ is equivalently described by
\[
Tg=\sigma g,
\qquad
u\in\widetilde E(g)+\mathcal K_0,
\qquad
\mathcal K_0=\ker(P_\Delta\partial_\nu|_\mathcal K).
\]
These compactness and Fredholm arguments concern only the compact comparison problem. The noncompact $Y$-noids are treated by the Fourier construction below.

\medskip
\noindent\emph{Fourier decomposition on the noncompact $Y$-noids.}
We now work on the actual $Y$-noids, with the space $\mathcal D(M)$ of Subsection \ref{subsec: index decomp}. Let $R_\Gamma$ be the radius of the junction. Use the same polar angle $\theta$ on all three faces, so $ds=R_\Gamma\,d\theta$ on $\Gamma$. On a catenoidal face use the coordinate $t\geq0$ in its rotational parametrization. On a planar face use the physical radius $\rho$, with junction $\rho=R_\Gamma$. The metric coefficients, $|A_i|^2$, and the weight $\omega$ are independent of $\theta$, and each $\beta_i$ is constant along $\Gamma$.

Use the real orthonormal basis of $L^2([0,2\pi],d\theta)$
\[
\phi_0=(2\pi)^{-1/2},\qquad
\phi_n^c=\pi^{-1/2}\cos(n\theta),\qquad
\phi_n^s=\pi^{-1/2}\sin(n\theta)\quad(n\geq1).
\]
The corresponding angular projections are denoted by $P_0$, $P_n^c$, and $P_n^s$. Componentwise,
\[
(P_0u)_i(s_i,\theta)
=\phi_0(\theta)\int_0^{2\pi}u_i(s_i,\vartheta)\phi_0(\vartheta)\,d\vartheta,
\]
and, for $n\ge1$,
\[
(P_n^cu)_i(s_i,\theta)
=\phi_n^c(\theta)\int_0^{2\pi}u_i(s_i,\vartheta)\phi_n^c(\vartheta)\,d\vartheta,
\]
\[
(P_n^su)_i(s_i,\theta)
=\phi_n^s(\theta)\int_0^{2\pi}u_i(s_i,\vartheta)\phi_n^s(\vartheta)\,d\vartheta.
\]
Here $s_i$ denotes $t$ or $\rho$ on the corresponding face. On the disk the projections act on $H^1$ of the whole disk. The center is not treated as a boundary component. Write
\[
\mathcal D_0=P_0\mathcal D(M),\qquad
\mathcal D_n^c=P_n^c\mathcal D(M),\qquad
\mathcal D_n^s=P_n^s\mathcal D(M),
\]
and define the Fourier truncation
\[
S_N=P_0+\sum_{n=1}^N(P_n^c+P_n^s).
\]

\begin{prop}\label{prop: second assumption steklov}
For every $Y$-noid, the spaces $\mathcal D_0$, $\mathcal D_n^c$, and $\mathcal D_n^s$ are closed. Distinct angular sectors are orthogonal for the $\mathcal D(M)$ inner product and for $Q$. For every $u,v\in\mathcal D(M)$,
\[
S_Nu\longrightarrow u\quad\text{in }\mathcal D(M),
\]
\[
\|u\|_{\mathcal D(M)}^2
=\|P_0u\|_{\mathcal D(M)}^2
+\sum_{n=1}^\infty\left(\|P_n^cu\|_{\mathcal D(M)}^2+\|P_n^su\|_{\mathcal D(M)}^2\right),
\]
and
\[
Q(u,v)=Q(P_0u,P_0v)
+\sum_{n=1}^\infty\left(Q(P_n^cu,P_n^cv)+Q(P_n^su,P_n^sv)\right),
\]
with absolute convergence of the last series. Each projection preserves trace compatibility and the facewise Jacobi equation.

Set
\[
\iota_0=\operatorname{Ind}(Q;\mathcal D_0),
\qquad
\iota_n=\operatorname{Ind}(Q;\mathcal D_n^c)=\operatorname{Ind}(Q;\mathcal D_n^s),\quad n\ge1.
\]
Then
\[
\operatorname{Ind}(Q;\mathcal D(M))
=\iota_0+2\sum_{n=1}^\infty\iota_n.
\]
The equality is understood in $\{0,1,2,\ldots,\infty\}$. The weak kernel decomposes as the completed orthogonal sum of the restricted bilinear kernels in these angular sectors.
\end{prop}

\begin{proof}
\noindent Step 1. \emph{Convergence in the variational norm.}
On a catenoidal face the Dirichlet energy has the form
\[
\int_0^\infty\int_0^{2\pi}\left(|u_t|^2+|u_\theta|^2\right)d\theta\,dt,
\]
while the weighted term is $\int\mu_i(t)|u|^2\,dt\,d\theta$ with $\mu_i$ independent of $\theta$. On a planar face the Dirichlet energy is
\[
\int\int\left(\rho|u_\rho|^2+\rho^{-1}|u_\theta|^2\right)d\rho\,d\theta.
\]
Parseval's identity applied to $u$, its weak radial derivative, and its angular derivative gives the displayed norm identity and $S_Nu\to u$ in $\mathcal D(M)$. On the disk, applying the argument first on the annuli $\varepsilon<\rho<R_\Gamma$ and then letting $\varepsilon\downarrow0$ gives the same conclusion because the $H^1$ energy of the omitted disk tends to zero.

Angular projection commutes with the Sobolev trace map on a collar of $\Gamma$. Hence the relation $u_1|_\Gamma+u_2|_\Gamma+u_3|_\Gamma=0$ is preserved mode by mode. The norm identity also shows that the angular projections are bounded and have closed ranges.

\medskip
\noindent Step 2. \emph{The full index form and the Jacobi equation.}
Since the metric coefficients, the potential $|A_i|^2$, the weight, and each $\beta_i$ are independent of $\theta$, distinct angular sectors are $Q$-orthogonal. Continuity of $Q$, together with Cauchy--Schwarz in the orthogonal decomposition, gives absolute convergence of the displayed expansion.

If $J_i u_i=0$ weakly and $P$ is one of $P_0,P_n^c,P_n^s$, then for every compactly supported test function $\psi$,
\[
\int_{M_i}\left(\langle\nabla(Pu_i),\nabla\psi\rangle-|A_i|^2(Pu_i)\psi\right)dA_i
=
\int_{M_i}\left(\langle\nabla u_i,\nabla(P\psi)\rangle-|A_i|^2u_i(P\psi)\right)dA_i=0.
\]
Hence each angular projection preserves the Jacobi equation.

\medskip
\noindent Step 3. \emph{Index and weak kernel.}
For $n\ge1$, rotation of the common angular variable by $\pi/(2n)$ maps the cosine sector isometrically onto the sine sector and preserves $Q$. Their indices are therefore equal.

Negative subspaces chosen from finitely many mutually orthogonal sectors give the lower bound in the index formula. For the upper bound, let $L$ be a $k$-dimensional negative subspace with basis $u_1,\ldots,u_k$. Since $S_Nu_a\to u_a$ in $\mathcal D(M)$,
\[
\bigl(Q(S_Nu_a,S_Nu_b)\bigr)_{a,b}
\longrightarrow
\bigl(Q(u_a,u_b)\bigr)_{a,b}.
\]
For large $N$ the left matrix remains negative definite. The projected vectors lie in finitely many angular sectors, where the form is block diagonal, so $k$ is at most the sum of the indices of those sectors. This gives the upper bound without assuming in advance that the total index is finite.

Finally, orthogonality gives $Q(Pu,v)=Q(u,Pv)$ for every angular projection $P$. Therefore each angular component of a full weak null field lies in the restricted bilinear kernel. Conversely, a vector in one restricted kernel pairs only with the corresponding angular component of an arbitrary test field, so it pairs by zero with all of $\mathcal D(M)$. Passing to norm-convergent sums gives the stated completed kernel decomposition.
\end{proof}

Proposition \ref{prop: second assumption steklov} reduces the full variational problem to the individual angular sectors. We now fix one frequency $n$ and integrate out the angular variable. Write $u_i=a_i\phi$ and $v_i=b_i\phi$, where $\phi=\phi_0$ for $n=0$ or either normalized angular function of frequency $n\geq1$. Then $\int_0^{2\pi}\phi^2\,d\theta=1$ and $\int_0^{2\pi}|\phi_\theta|^2\,d\theta=n^2$. Let $a_i^\Gamma,b_i^\Gamma$ be the radial traces at $\Gamma$. Their compatibility is $\sum_i a_i^\Gamma=\sum_i b_i^\Gamma=0$. On the circular junction, $ds=R_\Gamma d\theta$, so $\int_\Gamma\phi^2\,ds=R_\Gamma$. Substitution in $Q$ gives the following one-dimensional radial form.
\[
\begin{split}
Q(u,v)=\mathfrak q_n(a,b)
={}&\sum_{\text{catenoidal }i}\int_0^\infty
\left[a_i'b_i'+\left(n^2-\frac{2}{\cosh^2(t+T_i)}\right)a_ib_i\right]dt\\
&+\sum_{\text{planar }i}\int_{I_i}
\left[\rho a_i'b_i'+\frac{n^2}{\rho}a_ib_i\right]d\rho
-R_\Gamma\sum_{i=1}^3\beta_i a_i^\Gamma b_i^\Gamma.
\end{split}
\]
The radial domain consists precisely of the profiles whose lifts $a_i\phi$ belong to $\mathcal D(M_i)$ and satisfy the displayed compatibility. In particular, the disk lift belongs to $H^1$ across its center. For unnormalized $1$, $\cos(n\theta)$, or $\sin(n\theta)$, the entire displayed form has the additional factor $2\pi$ or $\pi$, respectively.

This decomposition applies to the full variational domain, including the resonant radial sector on a face with $T_i=0$, which is treated separately in the final subsection.

\medskip
\noindent\emph{Reduction within an angular sector.}
We now fix one Fourier sector and separate two effects. The subspace with zero junction trace is the fixed-junction Dirichlet part. The remaining part is determined by the compatible junction trace. In a nonresonant sector that trace has a unique admissible Jacobi extension on each face, so the sector splits into a zero-trace part and a finite-dimensional boundary part.

Fix one of the normalized angular functions $\phi$ above, of frequency $n\geq0$. Give
\[
\mathcal X_{i,n}=\{a_i:a_i\phi\in\mathcal D(M_i)\}
\]
the norm of its lift in $\mathcal D(M_i)$. Write $a_i^\Gamma$ for its junction trace and put
\[
\begin{split}
V_\Delta&=\{C\in\mathbb R^3:C_1+C_2+C_3=0\},\\
\mathcal X_n&=\{a\in\textstyle\bigoplus_i\mathcal X_{i,n}:
(a_1^\Gamma,a_2^\Gamma,a_3^\Gamma)\in V_\Delta\},\\
\mathcal X_{i,n}^0&=\{a_i\in\mathcal X_{i,n}:a_i^\Gamma=0\},
\qquad \mathcal X_n^0=\textstyle\bigoplus_i\mathcal X_{i,n}^0.
\end{split}
\]
These are closed spaces, since trace is continuous on a junction collar. Let $\mathfrak q_{i,n}$ denote the interior integral on face $i$ in the displayed formula for $\mathfrak q_n$, without its junction term, and set
\[
d_{i,n}=\operatorname{Ind}(\mathfrak q_{i,n};\mathcal X_{i,n}^0).
\]
We call a frequency nonresonant when $n\geq1$, or when $n=0$ and every catenoidal face has $T_i\ne0$. We call the radial sector resonant when $n=0$ and at least one catenoidal face has $T_i=0$. No restriction on $T_i$ is imposed for $n\geq1$.

For a nonresonant frequency, define a profile with junction value one on each face by
\[
p_{i,n}(t)=e^{-nt}\frac{n+\tanh(t+T_i)}{n+\tanh T_i}
\quad\text{on a catenoidal face},
\]
\[
p_{i,n}(\rho)=
\begin{cases}
(\rho/R_\Gamma)^n,&\text{on the disk},\\
(R_\Gamma/\rho)^n,&\text{on the exterior plane}.
\end{cases}
\]
Let $\delta_{i,n}$ be the outward conormal derivative of $p_{i,n}$ at the junction. Explicitly,
\[
\delta_{i,n}=
\begin{cases}
\displaystyle\frac{1}{R_\Gamma}
\left(n-\frac{1-\tanh^2T_i}{n+\tanh T_i}\right),
&\text{on a catenoidal face},\\[6pt]
\displaystyle\frac{n}{R_\Gamma},&\text{on a planar face}.
\end{cases}
\]
Define
\[
\mathscr E_n C=(C_ip_{i,n})_{i=1}^3,
\qquad
\mathfrak b_n(C,D)=R_\Gamma\sum_{i=1}^3
(\delta_{i,n}-\beta_i)C_iD_i
\quad(C,D\in V_\Delta).
\]
The map $\mathscr E_n$ is a prescribed-trace Jacobi extension for this frequency. It is distinct from the compact generalized extension $E$.

\begin{cor}\label{coro: general index decom}
For every nonresonant frequency,
\[
\mathcal X_n=\mathcal X_n^0\oplus_{\mathfrak q_n}\mathscr E_n(V_\Delta),
\qquad
\mathfrak q_n(\mathscr E_n C,a)
=R_\Gamma\sum_{i=1}^3(\delta_{i,n}-\beta_i)C_i a_i^\Gamma.
\]
The decomposition is direct, and its projections are bounded in the radial variational norm. In particular,
\[
\iota_n=\sum_{i=1}^3d_{i,n}
+\operatorname{Ind}(\mathfrak b_n;V_\Delta).
\]
Its bilinear kernel is
\[
\ker\mathfrak q_n
=\mathscr E_n(\ker\mathfrak b_n),
\qquad
C\in\ker\mathfrak b_n
\ \Longleftrightarrow\
(\delta_{i,n}-\beta_i)C_i=r_n\quad(i=1,2,3)
\]
for some real $r_n$, with $C\in V_\Delta$. For the corresponding two-dimensional field, the common boundary function is $r(\theta)=r_n\phi(\theta)$, so it need not be constant along $\Gamma$. Here $\ker\mathfrak q_n$ is tested against all of $\mathcal X_n$, and $\ker\mathfrak b_n$ against all of $V_\Delta$.

For every $Y$-noid, including one whose radial sector is resonant,
\[
\operatorname{Ind}(Q;\mathcal D(M))
=\iota_0+2\sum_{n=1}^{\infty}
\left(\sum_{i=1}^3d_{i,n}+\operatorname{Ind}(\mathfrak b_n;V_\Delta)\right).
\]
If every catenoidal $T_i$ is nonzero, this also gives
\[
\operatorname{Ind}(Q;\mathcal D(M))
=\sum_{i=1}^3 \operatorname{Ind}_{J_i}^0(M_i)
+\operatorname{Ind}(\mathfrak b_0;V_\Delta)
+2\sum_{n=1}^{\infty}\operatorname{Ind}(\mathfrak b_n;V_\Delta).
\]
Here $\operatorname{Ind}_{J_i}^0(M_i)$ is the index on the zero-junction-trace subspace of $\mathcal D(M_i)$, equivalently on its compactly supported functions. Sums of indices are understood in $\{0,1,2,\ldots,\infty\}$.
\end{cor}

\begin{proof}
\noindent Step 1. \emph{Admissible Jacobi profiles.}
On a catenoidal face the radial Jacobi equation is
\[
a''+\bigl(2\operatorname{sech}^2(t+T_i)-n^2\bigr)a=0.
\]
Writing $s=t+T_i$, two independent solutions are
\[
\begin{array}{c|c|c}
n&\text{first solution}&\text{second solution}\\ \hline
0&\tanh s&1-s\tanh s\\
1&\operatorname{sech}s&\sinh s+s\operatorname{sech}s\\
n\geq2&(n+\tanh s)e^{-ns}&(n-\tanh s)e^{ns}.
\end{array}
\]
These solve the equation by differentiation. Their Wronskians are nonzero, so they span every weak solution after the usual local regularity for this ordinary differential equation. There are constants $c_0,C_0>0$ such that, for all sufficiently large $t$,
\[
c_0(1+t)^{-2}\le \mu_i(t)=\omega(X_i(t,\theta))c_i^2\cosh^2(t+T_i)\le C_0(1+t)^{-2}.
\]
The first solution has finite Dirichlet energy and finite weighted norm. The second has derivative tending to $-1$ when $n=0$, and grows exponentially when $n\geq1$. In the latter case even the angular-energy integral $n^2\int a^2\,dt$ diverges. No combination with a nonzero second coefficient is admissible. Since $(1+\tanh s)e^{-s}=\operatorname{sech}s$, the normalized first solutions are exactly the displayed $p_{i,n}$.

On a planar face the radial equation is
$(\rho a')'-n^2\rho^{-1}a=0$. Its solutions are $1,\log\rho$ for $n=0$ and $\rho^n,\rho^{-n}$ for $n\geq1$. On the disk, finite energy at the center excludes $\log\rho$ and $\rho^{-n}$. On the exterior plane, finite energy excludes $\log\rho$ and $\rho^n$. The constant exterior solution has finite weighted norm because, for large $\rho$, $\omega(\rho)\rho$ is bounded above and below by positive multiples of $(\rho\log^2\rho)^{-1}$. Thus the displayed planar $p_{i,n}$ are again the unique admissible solutions with junction value one. The disk solutions extend smoothly across the center.

For each fixed surface and nonresonant frequency, the map $\mathscr E_n:V_\Delta\to\mathcal X_n$ is bounded. Differentiating the profiles gives the stated $\delta_{i,n}$, using $\partial_{\nu_i}=-R_\Gamma^{-1}\partial_t$ on a catenoidal face and $\partial_\nu=\partial_\rho$ or $-\partial_\rho$ on the disk or exterior plane, respectively.

\medskip
\noindent Step 2. \emph{Boundary pairing and directness.}
For a compactly supported radial test $a_i$, Green's formula on the face gives
\[
\mathfrak q_{i,n}(p_{i,n},a_i)
=R_\Gamma\delta_{i,n}a_i^\Gamma.
\]
On the disk this is Green's formula for the smooth lift $p_{i,n}\phi$ on the whole disk. On an unbounded face, use the radial ambient cutoffs from Proposition \ref{prop: function space for noncompact index}. Their restrictions are independent of $\theta$, equal one near $\Gamma$, and converge under multiplication in the single-face variational norm. The potential and trace estimates in that proposition give continuity of the single-face form, so passing to the limit proves the identity for every $a_i\in\mathcal X_{i,n}$ without introducing a boundary term at infinity.

Summing and subtracting the junction term proves the pairing formula in the corollary. For $a\in\mathcal X_n$, let $C=(a_i^\Gamma)_i$ and $w=a-\mathscr E_n C$. Then $w\in\mathcal X_n^0$. The pairing formula gives $\mathfrak q_n(w,\mathscr E_n D)=0$ for all $D\in V_\Delta$. A zero sum has trace $C=0$, and hence $w=0$. Trace continuity and the bounded finite-dimensional extension give the bounded decomposition projections.

\medskip
\noindent Step 3. \emph{Index and bilinear kernel.}
On $\mathcal X_n^0$ the form is the direct sum of the three zero-trace face forms, while on $\mathscr E_n(V_\Delta)$ it is $\mathfrak b_n$. Finite negative subspaces in these orthogonal summands give the lower bound. For the upper bound, the projection images of any finite-dimensional negative subspace lie in a direct sum on which the form is block diagonal, so its index is at most the sum of the indices of the two full summands. This gives the sector index formula without assuming positivity of the zero-trace part.

The bilinear kernel of each zero-trace face form is trivial. Indeed, interior tests make such a field Jacobi. Step 1 shows that the admissible radial Jacobi solution space in a nonresonant face is one-dimensional and is spanned by $p_{i,n}$. Since a zero-trace field has junction value zero while $p_{i,n}$ has junction value one, the coefficient must vanish. Orthogonality now gives $\ker\mathfrak q_n=\mathscr E_n(\ker\mathfrak b_n)$. A vector $C\in V_\Delta$ is in $\ker\mathfrak b_n$ exactly when $((\delta_{i,n}-\beta_i)C_i)_i$ is orthogonal to $V_\Delta$, or equivalently has three equal components. This is the full coupled junction condition for the chosen angular function.

Finally apply Proposition \ref{prop: second assumption steklov}. Its proof applies on each single face with zero junction trace as well: the angular projections and cutoffs preserve that trace. Thus
$\operatorname{Ind}_{J_i}^0(M_i)=d_{i,0}+2\sum_{n\geq1}d_{i,n}$, and the single-face cutoff argument identifies this with the compactly supported Dirichlet index. Substitution yields the two global formulas.
\end{proof}

In the boundary inner product $R_\Gamma C\cdot D$ on $V_\Delta$, the form $\mathfrak b_n$ is represented by the symmetric operator
\[
S_n C=P_\Delta\bigl((\delta_{i,n}-\beta_i)C_i\bigr)_{i=1}^3,
\qquad
P_\Delta C=C-\tfrac13(C_1+C_2+C_3)(1,1,1).
\]
The equation $S_nC=\sigma C$ is equivalent to
$(\delta_{i,n}-\beta_i-\sigma)C_i=r_n$ for all $i$. The numbers $\delta_{i,n}$ are facewise conormal coefficients rather than eigenvalues of the coupled operator, and the corresponding common boundary function is $r(\theta)=r_n\phi(\theta)$.

If $n=0$ and $T_i=0$ on a catenoidal face, the only admissible radial Jacobi fields there are multiples of $\tanh t$, with zero junction trace. There is no admissible Jacobi extension of a nonzero constant trace. The corollary therefore leaves $\iota_0=\operatorname{Ind}(\mathfrak q_0;\mathcal X_0)$ unsplit in this case. It includes that radial sector in the global index rather than assigning it a Dirichlet-to-Neumann coefficient or a compact generalized extension.

We now justify the noncompact variational framework used in Proposition~\ref{prop: second assumption steklov}, Corollary~\ref{coro: general index decom}, and the computations of Section~\ref{section 4}.

\begin{prop}[Noncompact variational framework]\label{prop: function space for noncompact index}
Let $M$ be one of the $Y$-noids considered in this paper, and let $\mathcal D(M)$ and $L_*^2(M)$ be as in Subsection~\ref{subsec: index decomp}. Then the following hold.
\begin{enumerate}
\item[(a)] $\mathcal D(M)$ is a Hilbert space, dense in $L_*^2(M)$. The index form $Q$ is continuous on $\mathcal D(M)$ and defines a closed, lower-semibounded form on $L_*^2(M)$.

\item[(b)] Compactly supported admissible $H^1$ functions are dense in $\mathcal D(M)$. Consequently,
\[
\operatorname{Ind}(Q;\mathcal D(M))
=
\operatorname{Ind}(Q;W_{\Delta,c}^{1,2}(M)).
\]
This equality concerns the Morse index only. The cutoff argument does not preserve the bilinear kernel.

\item[(c)] Let $\mathscr A_*$ be the self-adjoint operator associated with $Q$ on the weighted Hilbert space $L_*^2(M)$. Writing $f_i=u_i|_\Gamma$, its domain is
\[
\begin{split}
\operatorname{Dom}(\mathscr A_*)=\biggl\{u\in\mathcal D(M):\;&
\sum_{i=1}^3\int_{M_i}\omega^{-1}|J_i u_i|^2\,dA_i<\infty,\\
&\partial_{\nu_i}u_i-\beta_i f_i=r\quad(i=1,2,3)\\
&\text{for some }r\in H^{-1/2}(\Gamma)\biggr\}.
\end{split}
\]
On this domain,
\[
\mathscr A_*u=(-\omega^{-1}J_i u_i)_{i=1}^3.
\]
In particular,
\[
\ker\mathscr A_*=\ker_{\mathcal D}Q,
\]
and $u\in\ker_{\mathcal D}Q$ if and only if
\[
J_i u_i=0\quad\text{on }M_i,\qquad
\sum_{i=1}^3 f_i=0,\qquad
\partial_{\nu_i}u_i-\beta_i f_i=r\quad\text{on }\Gamma
\]
for one common $r\in H^{-1/2}(\Gamma)$.

\item[(d)] A nonzero-trace field $u\in\mathcal D(M)$ satisfies the weak coupled Steklov equation
\[
Q(u,v)=\sigma\sum_{i=1}^3\int_\Gamma f_i g_i\,ds
\qquad(v\in\mathcal D(M)),
\]
where $g_i=v_i|_\Gamma$, if and only if
\[
J_i u_i=0\quad\text{on }M_i,\qquad
\sum_{i=1}^3f_i=0,
\qquad
\partial_{\nu_i}u_i-\beta_i f_i-\sigma f_i=r\quad\text{on }\Gamma
\]
with the same $r\in H^{-1/2}(\Gamma)$ for all three faces.
\end{enumerate}
\end{prop}

\begin{proof}
\noindent Step 1. \emph{Completeness and continuity of the form.}
Let $(u^{(j)})$ be Cauchy in $\mathcal D(M)$. It is Cauchy in $L_*^2(M)$, and its gradients are Cauchy in $L^2$ on each face. On every compact subset of a face, including a collar of $\Gamma$, the weight $\omega$ has a positive lower bound. Hence $(u^{(j)})$ is Cauchy in the ordinary $H^1$ norm there. The local $H^1$ limits agree with the $L_*^2$ limit, their weak gradients are the $L^2$ limits of $\nabla u^{(j)}$, and continuity of the trace map passes the condition $\sum_i u_i|_\Gamma=0$ to the limit. Thus $\mathcal D(M)$ is complete. Smooth triples compactly supported in the interiors of the faces have zero junction trace and are dense in $L_*^2(M)$, so $\mathcal D(M)$ is dense in $L_*^2(M)$.

For a catenoidal face parameterized as in Proposition~\ref{prop: kernel(J) of cat-faces},
\[
|A|^2=\frac{2}{c^2\cosh^4(t+T)},
\qquad
|x|^2=c^2\bigl(\cosh^2(t+T)+t^2\bigr),
\]
so $|A|^2=O(|x|^{-4})$ at infinity. On a planar face $|A|^2=0$. Therefore, for each fixed $Y$-noid,
\[
|A_{\mathbf F(M_i)}|^2\le C\omega\quad\text{on }M_i.
\]
On fixed compact collars of $\Gamma$, the trace theorem and the local interpolation inequality give, for every $\varepsilon>0$,
\[
\sum_{i=1}^3\|u_i|_\Gamma\|_{L^2(\Gamma)}^2
\le
\varepsilon\sum_{i=1}^3\|\nabla u_i\|_{L^2(M_i)}^2
+C_\varepsilon\|u\|_{L_*^2(M)}^2.
\]
The potential estimate and this trace estimate imply
\[
|Q(u,v)|\le C\|u\|_{\mathcal D(M)}\|v\|_{\mathcal D(M)}.
\]
They also give, after choosing $C_1$ sufficiently large,
\[
c\|u\|_{\mathcal D(M)}^2
\le Q(u,u)+C_1\|u\|_{L_*^2(M)}^2
\le C\|u\|_{\mathcal D(M)}^2.
\]
Since $\mathcal D(M)$ is complete, the shifted form norm is complete. Hence $Q$ is closed and lower semibounded on $L_*^2(M)$.

\medskip
\noindent Step 2. \emph{Cutoff approximation and the Morse index.}
Choose $R>2$ large enough that $\Gamma\subset B_R$, and let $\eta_R$ be a common smooth ambient cutoff satisfying
\[
0\le\eta_R\le1,
\qquad
\eta_R=1\text{ on }B_R,
\qquad
\eta_R=0\text{ outside }B_{R^2},
\]
and
\[
|\nabla_{M_i}\eta_R|
\le\frac{C}{|x|\log R}
\quad\text{on }B_{R^2}\setminus B_R.
\]
Such cutoffs are obtained by smoothing a function of $\log|x|/\log R$. See \cite[Section 2.6]{chodosh2018topology}. The same ambient cutoff is used on all three faces, so the compatibility condition along $\Gamma$ is preserved. Since the faces are properly immersed, $\eta_Ru$ has compact support on the parameterized surface.

On the transition region $R\le |x|\le R^2$, the quantities $\log(2+|x|)$ and $\log R$ are comparable. Consequently,
\[
|\nabla_{M_i}\eta_R|^2\le C\omega,
\]
and therefore
\[
\sum_{i=1}^3\int_{M_i}u_i^2|\nabla\eta_R|^2\,dA_i
\le
C\sum_{i=1}^3\int_{M_i\setminus B_R}\omega u_i^2\,dA_i
\longrightarrow0.
\]
Together with the tails of $\int|\nabla u|^2$ and $\int\omega u^2$, this gives
\[
\eta_Ru\longrightarrow u\qquad\text{in }\mathcal D(M).
\]
Thus compactly supported admissible $H^1$ functions are dense in $\mathcal D(M)$.

Now let $L\subset\mathcal D(M)$ be a $k$-dimensional subspace on which $Q$ is negative definite, with basis $\psi_1,\ldots,\psi_k$. By continuity of $Q$,
\[
\bigl(Q(\eta_R\psi_j,\eta_R\psi_\ell)\bigr)_{j,\ell}
\longrightarrow
\bigl(Q(\psi_j,\psi_\ell)\bigr)_{j,\ell}.
\]
For $R$ sufficiently large, the matrix on the left is still negative definite. Hence the cutoff functions are linearly independent and span a $k$-dimensional negative subspace of $W_{\Delta,c}^{1,2}(M)$. This proves
\[
\operatorname{Ind}(Q;\mathcal D(M))
\le
\operatorname{Ind}(Q;W_{\Delta,c}^{1,2}(M)).
\]
The reverse inequality follows from $W_{\Delta,c}^{1,2}(M)\subset\mathcal D(M)$. The argument uses stability of strict negative definiteness under small perturbations. It does not imply preservation of the bilinear kernel.

\medskip
\noindent Step 3. \emph{The weighted operator and its weak kernel.}
By the representation theorem for closed lower-semibounded forms, $Q$ determines a self-adjoint operator $\mathscr A_*$ on $L_*^2(M)$. By definition,
\[
\begin{split}
\operatorname{Dom}(\mathscr A_*)
=\bigl\{u\in\mathcal D(M):\;&\text{there is }z\in L_*^2(M)\text{ such that}\\
&Q(u,v)=\sum_{i=1}^3\int_{M_i}\omega z_i v_i\,dA_i
\text{ for every }v\in\mathcal D(M)\bigr\}.
\end{split}
\]
with $\mathscr A_*u=z$. Interior tests give
\[
J_i u_i=-\omega z_i
\]
in the distributional sense. Thus $J_i u_i$ is locally $L^2$ up to $\Gamma_i$, and
\[
\mathscr A_*u=(-\omega^{-1}J_i u_i)_{i=1}^3.
\]
In particular, a weighted volume eigenfunction satisfies $-J_i u_i=\lambda\omega u_i$.

Suppose first that $u\in\operatorname{Dom}(\mathscr A_*)$. The preceding identity gives
\[
\sum_{i=1}^3\int_{M_i}\omega^{-1}|J_i u_i|^2\,dA_i<\infty.
\]
Green's formula on compact boundary collars and tests whose traces are $(a,-a,0)$ and $(a,0,-a)$ show that there is one $r\in H^{-1/2}(\Gamma)$ such that
\[
\partial_{\nu_i}u_i-\beta_i f_i=r,
\qquad i=1,2,3.
\]
Conversely, assume the displayed integrability and common-flux condition. Put $z_i=-\omega^{-1}J_i u_i$, so $z\in L_*^2(M)$. Green's formula against a compactly supported admissible $H^1$ function $v$, with $g_i=v_i|_\Gamma$, gives
\[
Q(u,v)
=
\sum_{i=1}^3\int_{M_i}\omega z_i v_i\,dA_i
+
\left\langle r,\sum_{i=1}^3g_i\right\rangle_\Gamma
=
\sum_{i=1}^3\int_{M_i}\omega z_i v_i\,dA_i.
\]
Both sides are continuous in the $\mathcal D(M)$ norm. Step 2 therefore extends the identity to all $v\in\mathcal D(M)$, proving the domain characterization.

Taking $z=0$ gives the description of $\ker\mathscr A_*$. Since
\[
Q(u,v)=0\quad\text{for every }v\in\mathcal D(M)
\]
is exactly the condition $u\in\ker\mathscr A_*$, we have $\ker\mathscr A_*=\ker_{\mathcal D}Q$ and obtain the stated weak-kernel characterization.

\medskip
\noindent Step 4. \emph{The coupled Steklov problem.}
Let $u\in\mathcal D(M)$ have nonzero trace $f=(f_i)$ and suppose
\[
Q(u,v)=\sigma\sum_{i=1}^3\int_\Gamma f_i g_i\,ds
\qquad(v\in\mathcal D(M)).
\]
Tests supported in the interior of one face give $J_i u_i=0$. Using Green's formula and boundary collar tests with compatible traces $(a,-a,0)$ and $(a,0,-a)$ then shows that
\[
\partial_{\nu_i}u_i-\beta_i f_i-\sigma f_i=r
\qquad(i=1,2,3)
\]
for one common $r\in H^{-1/2}(\Gamma)$. Conversely, if the facewise Jacobi equations, trace compatibility, and this common-flux condition hold, Green's formula gives the weak Steklov identity for compactly supported admissible tests, and Step 2 extends it to every $v\in\mathcal D(M)$.

The weighted volume problem and the coupled Steklov problem therefore have the same distinction as in the compact setting. Trace compatibility determines the form domain. The operator equation imposes one common boundary flux on the three faces.
\end{proof}

\section{Morse Index and Nullity of the \texorpdfstring{$Y$}{Y}-Noid Family} \label{section 4}

\begin{prop} \label{prop: kernel(J) of cat-faces}
Let
\[
X(t,\theta)=c\big(\cosh(t+T)\cos\theta,\cosh(t+T)\sin\theta,\pm t\big),
\qquad t\in[0,\infty),
\]
parameterize a catenoidal end $\Sigma$, with $T\in\mathbb R$. Define
\[
\ker_{\mathcal D}J(\Sigma)=\{u\in\mathcal D(\Sigma):Ju=0\}.
\]
Then
\[
\ker_{\mathcal D}J(\Sigma)
=\overline{\bigoplus_{n\ge0}W_n}^{\,\mathcal D(\Sigma)},
\]
where
\[
\begin{aligned}
W_0&=\operatorname{span}\{\tanh(t+T)\},\\
W_1&=\operatorname{span}\left\{\frac{\cos\theta}{\cosh(t+T)},\frac{\sin\theta}{\cosh(t+T)}\right\},\\
W_n&=\operatorname{span}\left\{\cos(n\theta)(n+\tanh(t+T))e^{-n(t+T)},\right.\\
&\hspace{3.1cm}\left.\sin(n\theta)(n+\tanh(t+T))e^{-n(t+T)}\right\},\quad n\ge2.
\end{aligned}
\]
For $T\neq0$, the restrictions of these functions to $\{t=0\}$ are facewise Steklov eigenfunctions with
\[
\delta_0=-\frac1{c\cosh^2T\sinh T},\qquad
\delta_1=\frac{\sinh T}{c\cosh^2T},
\]
and
\[
\delta_n=-\frac{1-(n\cosh^2T)(n+\tanh T)}
{c(n+\tanh T)\cosh^3T},\qquad n\ge2.
\]
When $T=0$, the radial field $\tanh t$ has zero trace on the junction and hence is a Dirichlet Jacobi field rather than a nonzero Steklov trace.
\end{prop}

\begin{proof}
The induced metric and Jacobi operator are
\[
g_\Sigma=c^2\cosh^2(t+T)(dt^2+d\theta^2),\qquad
J=\frac1{c^2\cosh^2(t+T)}\left(\partial_t^2+\partial_\theta^2+2\operatorname{sech}^2(t+T)\right),
\]
and the outward conormal derivative at $t=0$ is
\[
D_\nu=-\frac1{c\cosh T}\,\partial_t.
\]
By the single-face Fourier convergence proved in Proposition \ref{prop: second assumption steklov}, every $u\in\mathcal D(\Sigma)$ is the $\mathcal D(\Sigma)$-limit of its angular Fourier partial sums, and the Jacobi equation is preserved by the angular projections. Thus it is enough to solve the radial equation in each frequency.

For a mode $f(t)(a\cos n\theta+b\sin n\theta)$, the radial equation is
\[
f''+\big(2\operatorname{sech}^2(t+T)-n^2\big)f=0.
\]
For $n=0$ two independent solutions are
\[
\tanh(t+T),\qquad 1-(t+T)\tanh(t+T).
\]
The second has derivative tending to a nonzero constant and therefore has infinite Dirichlet energy. The first belongs to $\mathcal D(\Sigma)$. For $n=1$ two independent solutions are
\[
\operatorname{sech}(t+T),\qquad
\sinh(t+T)+(t+T)\operatorname{sech}(t+T),
\]
and only the first is admissible. For $n\ge2$ two independent solutions are
\[
(n+\tanh(t+T))e^{-n(t+T)},\qquad
(n-\tanh(t+T))e^{n(t+T)},
\]
and only the decaying solution belongs to $\mathcal D(\Sigma)$. This proves that every Fourier component of a field in $\ker_{\mathcal D}J(\Sigma)$ lies in the corresponding $W_n$, while every element of $W_n$ is an admissible Jacobi field. Conversely, finite sums of elements of the $W_n$ are Jacobi fields. If such sums converge in $\mathcal D(\Sigma)$, they converge locally in $H^1$ on compact subsets, so the weak Jacobi equation passes to the limit. The Fourier convergence therefore gives the asserted closure in the $\mathcal D(\Sigma)$-norm.

For $T\neq0$, dividing the outward conormal derivative of each admissible profile by its nonzero boundary value gives the displayed coefficients $\delta_0,\delta_1,\delta_n$. These are single-face Dirichlet-to-Neumann coefficients. They are not by themselves eigenvalues of the coupled triple-junction problem. When $T=0$, the radial profile is $\tanh t$ and vanishes at the junction, giving the stated exception.
\end{proof}
\medskip

\begin{prop} \label{prop: Ker(J) of flat-faces}
Let $\rho$ denote the physical Euclidean radius.  Parameterize the closed disk and the exterior planar face by
\[
X_1(\rho,\theta)=(\rho\cos\theta,\rho\sin\theta,0),
\qquad 0\le \rho\le R_0,
\]
and
\[
X_2(\rho,\theta)=(\rho\cos\theta,\rho\sin\theta,0),
\qquad R_0\le \rho<\infty.
\]
Thus the common boundary is $\Gamma=\{\rho=R_0\}$.  On the disk, the admissible Jacobi fields are the $\mathcal D$-closure of the linear combinations of
\[
w_0=1,\qquad
w_n(\rho,\theta)=
\left(a_n\cos n\theta+b_n\sin n\theta\right)
\left(\frac{\rho}{R_0}\right)^n,
\qquad n\ge1.
\]
On the exterior plane, the admissible Jacobi fields are the $\mathcal D$-closure of the linear combinations of
\[
w_0=1,\qquad
w_n(\rho,\theta)=
\left(a_n\cos n\theta+b_n\sin n\theta\right)
\left(\frac{R_0}{\rho}\right)^n,
\qquad n\ge1.
\]
In both cases the boundary traces are
\[
w_0|_\Gamma=1,\qquad
w_n|_\Gamma=a_n\cos n\theta+b_n\sin n\theta,
\]
and their outward conormal coefficients are
\[
\delta_0=0,\qquad \delta_n=\frac{n}{R_0},\qquad n\ge1.
\]
\end{prop}
\begin{proof}
On a planar face $J=\Delta$. Separation of variables gives the radial harmonic functions $1$ and $\log\rho$ for $n=0$, and $\rho^n$ and $\rho^{-n}$ for $n\ge1$.

On the disk, $\log\rho$ and $\rho^{-n}$ have infinite Dirichlet energy at the origin, whereas $1$ and $\rho^n$ are admissible. Normalizing the positive-frequency solution to have boundary value one at $\rho=R_0$ gives $(\rho/R_0)^n$.

On the exterior plane,
\[
\int_{R_0}^\infty|\partial_\rho\log\rho|^2\rho\,d\rho
=\int_{R_0}^\infty\frac{d\rho}{\rho}=\infty,
\]
so the logarithmic solution is excluded by finite Dirichlet energy. For $n\ge1$,
\[
\int_{R_0}^\infty|\partial_\rho(\rho^n)|^2\rho\,d\rho=\infty,
\qquad
\int_{R_0}^\infty|\partial_\rho(\rho^{-n})|^2\rho\,d\rho<\infty.
\]
The decaying mode also lies in $L_*^2$. The constant mode is admissible because its Dirichlet energy is zero and
\[
\int_{R_0}^\infty
\frac{\rho\,d\rho}{(1+\rho^2)\log^2(2+\rho)}<\infty.
\]
Thus the admissible exterior modes are exactly those stated, and normalization at the junction gives $(R_0/\rho)^n$.

On the disk the outward conormal is $\partial_\rho$, while on the exterior plane it is $-\partial_\rho$. Hence
\[
\partial_\nu\left(\frac{\rho}{R_0}\right)^n\Big|_\Gamma
=\frac n{R_0},
\qquad
-\partial_\rho\left(\frac{R_0}{\rho}\right)^n\Big|_\Gamma
=\frac n{R_0},
\]
and the constant mode has coefficient zero.

Finally, Proposition \ref{prop: second assumption steklov} gives convergence of the angular Fourier series in the corresponding $\mathcal D$ norm and preserves the Jacobi equation under angular projection. Hence every admissible harmonic field is a $\mathcal D$-limit of the stated modes. Conversely, finite sums of these modes are harmonic, and local $H^1$ convergence passes the weak equation to every $\mathcal D$-limit. The coefficients $\delta_n$ are facewise Dirichlet-to-Neumann coefficients. The coupled triple-junction condition is imposed only after the three faces are assembled.
\end{proof}

\subsection{The \texorpdfstring{$Y$}{Y}-Catenoid}
\begin{thm}
The Morse index of the $Y$-catenoid is one, and the nullity is three.
\end{thm}
\begin{proof}
Let $M=(\Sigma_1,\Sigma_2,\Sigma_3;\Gamma)=Y_0$. The disk $\Sigma_1$ is flat, and the two catenoidal faces have parameter $T=\log\sqrt3>0$. A positive radial Jacobi field on each catenoidal face is $\tanh(t+T)$. The ground-state transform therefore gives nonnegativity of the zero-junction-trace Dirichlet form. Hence
\[
\sum_{i=1}^3\operatorname{Ind}_{J_i}^0(\Sigma_i)=0.
\]
The junction radius is $R_\Gamma=2/\sqrt3$, and
\[
(\beta_1,\beta_2,\beta_3)
=\left(-\frac{\sqrt3}{2},\frac{\sqrt3}{4},\frac{\sqrt3}{4}\right).
\]
All catenoidal parameters are nonzero, so Corollary \ref{coro: general index decom} applies to every angular frequency.

For $n=0$, Propositions \ref{prop: kernel(J) of cat-faces} and \ref{prop: Ker(J) of flat-faces} give
\[
(\delta_{1,0}-\beta_1,\delta_{2,0}-\beta_2,\delta_{3,0}-\beta_3)
=\left(\frac{\sqrt3}{2},-\sqrt3,-\sqrt3\right).
\]
On $V_\Delta=\{C_1+C_2+C_3=0\}$,
\[
\mathfrak b_0(C,D)=-(C_2-C_3)(D_2-D_3),
\]
so
\[
\mathfrak b_0(C,C)=-(C_2-C_3)^2.
\]
Thus $\mathfrak b_0$ has exactly one negative direction and
\[
\ker\mathfrak b_0=\operatorname{span}\{(-2,1,1)\}.
\]
For the unnormalized constant angular function $1$, the full index form is $2\pi\mathfrak b_0$. Hence
\[
Q(u,u')=-2\pi(C_2-C_3)(C_2'-C_3')
\]
for constant trace vectors $C,D\in V_\Delta$. In particular, the trace vector $(0,1,-1)$ has value $-8\pi$, with $|\Gamma|=2\pi R_\Gamma=4\pi/\sqrt3$.

The kernel vector satisfies the coupled condition because
\[
(\delta_{i,0}-\beta_i)C_i=-\sqrt3
\qquad(i=1,2,3).
\]
Hence the radial boundary sector contributes one negative direction and one weak-kernel direction.

For $n=1$,
\[
(\delta_{1,1}-\beta_1,\delta_{2,1}-\beta_2,\delta_{3,1}-\beta_3)=(\sqrt3,0,0).
\]
Therefore
\[
\mathfrak b_1(C,C)=2C_1^2,
\qquad
\ker\mathfrak b_1=\operatorname{span}\{(0,1,-1)\}.
\]
There is one such kernel direction in the cosine sector and one in the sine sector, and no negative direction.

For $n\ge2$, direct simplification of the facewise coefficients gives
\[
\delta_{1,n}-\beta_1=\frac{\sqrt3}{2}(n+1)>0,
\qquad
\delta_{2,n}-\beta_2=\delta_{3,n}-\beta_3
=\frac{\sqrt3(n^2-1)}{2n+1}>0.
\]
Hence every higher boundary form is positive definite on $V_\Delta$.

Corollary \ref{coro: general index decom} now gives the exact index accounting
\[
\operatorname{Ind}(Y_0)
=\underbrace{0}_{\text{zero-trace Dirichlet part}}
+\underbrace{1}_{n=0\text{ boundary part}}
+\underbrace{0}_{n\ge1\text{ boundary part}}
=1.
\]
The weak nullity is
\[
\operatorname{nul}_{\mathcal D}(Y_0)
=\underbrace{1}_{n=0}
+\underbrace{1}_{\cos\theta}
+\underbrace{1}_{\sin\theta}
=3.
\]
The sector kernel formula in Corollary \ref{coro: general index decom} rules out additional zero-trace or higher-mode weak-kernel directions.
\end{proof}

\begin{rem}[Geometric meaning of the nullity] \label{rem: geometric meaning ycat}
\rm
The three null directions are the normal components of ambient translations. With compatible orientations, vertical translation has junction trace proportional to $(-2,1,1)$ and gives the radial kernel line. The two horizontal translations vanish on the disk and have traces proportional to $(0,1,-1)\cos\theta$ and $(0,1,-1)\sin\theta$. These are exactly the two $n=1$ kernel directions. Since the weak nullity is three, the translation fields exhaust the weak kernel.
\end{rem}
%%%%%%%%%%%%%%%%%%%%%%%%%%%%%%%%%%%%%%%%%%%%%%%%%%%%%%%%%%%%%%%%%%%%%%%%%%%%%%%%%%%%%%%%%%%%%%%%%%%%%%%%%%%%%%%%%%

\subsection{The Pseudo \texorpdfstring{$Y$}{Y}-Catenoid} \label{sec: index of pseudo y-cat}
The two catenoidal faces are parametrized by
\[
X_1,X_3(t,\theta)=\bigl(\cosh(t-\log\sqrt3)\cos\theta,\cosh(t-\log\sqrt3)\sin\theta,\pm t\bigr),
\qquad t\ge0,
\]
and the exterior planar face by
\[
X_2(t,\theta)=\left(\left(t+\frac2{\sqrt3}\right)\cos\theta,\left(t+\frac2{\sqrt3}\right)\sin\theta,0\right),
\qquad t\ge0.
\]

\begin{thm} \label{thm: index pseudo y-cat}
The Morse index of the pseudo $Y$-catenoid is two, and the nullity is five.
\end{thm}
\begin{proof}
The two catenoidal faces have $T=-\log\sqrt3$, while the exterior planar face is Dirichlet stable. For a catenoidal face, a negative radial Dirichlet eigenvalue $-\mu^2$ has the unique decaying profile
\[
(\mu+\tanh(t+T))e^{-\mu(t+T)},
\qquad \mu>0.
\]
The zero boundary condition requires $\mu=-\tanh T$. Since $\tanh(-\log\sqrt3)=-1/2$, each catenoidal face has exactly one negative radial Dirichlet direction. For $n=1$ and zero junction trace,
\[
\mathfrak q_{i,1}(a,a)=\int_0^\infty
\left|a'+\tanh(t+T)a\right|^2dt\ge0,
\]
while the $n\ge2$ forms are positive. Hence
\[
\sum_{i=1}^3\operatorname{Ind}_{J_i}^0(\Sigma_i)=1+0+1=2.
\]

The junction radius is $R_\Gamma=2/\sqrt3$, and
\[
(\beta_1,\beta_2,\beta_3)
=\left(-\frac{\sqrt3}{4},\frac{\sqrt3}{2},-\frac{\sqrt3}{4}\right).
\]
All catenoidal parameters are nonzero, so Corollary \ref{coro: general index decom} applies to every frequency.

For $n=0$,
\[
(\delta_{1,0}-\beta_1,\delta_{2,0}-\beta_2,\delta_{3,0}-\beta_3)
=\left(\sqrt3,-\frac{\sqrt3}{2},\sqrt3\right).
\]
On $V_\Delta$ this gives
\[
\mathfrak b_0(C,D)=(C_1-C_3)(D_1-D_3).
\]
Thus
\[
\mathfrak b_0(C,C)=(C_1-C_3)^2\ge0,
\qquad
\ker\mathfrak b_0=\operatorname{span}\{(1,-2,1)\}.
\]
The kernel vector satisfies the coupled condition because
\[
\sqrt3(1)=\left(-\frac{\sqrt3}{2}\right)(-2)=\sqrt3(1).
\]
The complementary direction $(1,0,-1)$ has positive value $4$. Therefore the radial boundary form contributes no index and exactly one weak-kernel direction.

For $n=1$,
\[
\delta_{1,1}=\delta_{3,1}=-\frac{\sqrt3}{4},
\qquad
\delta_{2,1}=\frac{\sqrt3}{2},
\]
so $\delta_{i,1}-\beta_i=0$ on all three faces. Hence $\mathfrak b_1\equiv0$ on $V_\Delta$. The cosine and sine sectors contribute two weak-kernel dimensions each and no negative directions.

For $n\ge2$,
\[
\delta_{1,n}-\beta_1=\delta_{3,n}-\beta_3
=\frac{\sqrt3(n^2-1)}{2n-1}>0,
\qquad
\delta_{2,n}-\beta_2=\frac{\sqrt3}{2}(n-1)>0.
\]
Thus all higher boundary forms are positive definite. Corollary \ref{coro: general index decom} also gives trivial zero-trace bilinear kernels in these nonresonant sectors.

The final accounting is therefore
\[
\operatorname{Ind}(Y_{\pi/3})
=\underbrace{2}_{\text{zero-trace Dirichlet part}}
+\underbrace{0}_{n=0\text{ boundary part}}
+\underbrace{0}_{n\ge1\text{ boundary part}}
=2,
\]
and
\[
\operatorname{nul}_{\mathcal D}(Y_{\pi/3})
=\underbrace{1}_{n=0}
+\underbrace{2}_{\cos\theta}
+\underbrace{2}_{\sin\theta}
=5.
\]
\end{proof}

\begin{rem}[Geometric meaning of the nullity] \label{rem: geometric meaning pseudo ycat}
\rm
The radial kernel is generated by vertical translation. In the $n=1$ kernel, the subspace whose component on the planar face vanishes is exactly the two-dimensional space generated by horizontal translations. In each of the cosine and sine sectors, the complementary trace line may be taken to be $\operatorname{span}\{(1,-2,1)\}$. These fields have a nonzero component on the planar face and are not generated by translations. No nontrivial ambient rotation gives another admissible field in $\mathcal D(M)$. Horizontal rotations grow linearly on the exterior plane and have infinite Dirichlet energy. Rotation about the vertical axis has zero normal component.
\end{rem}

%%%%%%%%%%%%%%%%%%%%%%%%%%%%%%%%%%%%%%%%%%%%%%%%%%%%%%%%%%%%%%%%%%%%%%%%%%%%%%%%%%%%%%%%%%%%%%%%%%%%%%%%%%%%%%%%%%

\subsection{The Remaining \texorpdfstring{$Y$}{Y}-Noids} \label{sec: index of rest of family}
Let $M=(\Sigma_1,\Sigma_2,\Sigma_3;\Gamma)=Y_\alpha$ with $0<\alpha<\pi/3$. Away from the endpoints $Y_0$ and $Y_{\pi/3}$, all three faces are catenoidal. We write
\begin{equation}
\label{eqn: y-noids family maps}
\nonumber
X_i (t,\theta ) = c_i \big( \cosh  (t  + T_i ) \cos \theta ,  \cosh   (t + T_i ) \sin \theta ,    \pm t \big), \qquad t\in [0,\infty),  
\end{equation}
for $i = 1, 2, 3$. 
\medskip

Each $X_i$ parametrizes $\Sigma_i$, and the three faces meet along the circle $\Gamma\subset\{x_3=0\}$. Relative to the horizontal projection, a catenoidal face is either graphical or non-graphical. The $120^\circ$ conormal condition excludes the cases in which all three faces have the same type.

Up to reflection, it suffices to consider $\alpha\in[0,\pi/3]$. For $0<\alpha<\pi/6$, the surface has one non-graphical face and two graphical faces. For $\pi/6<\alpha<\pi/3$, it has two non-graphical faces and one graphical face. At $\alpha=\pi/6$ there is one non-graphical face and one catenoidal face with $T_i=0$. We treat the two open subfamilies separately and consider $Y_{\pi/6}$ at the end of the section.
\medskip

\begin{lem}\label{lem: dirichlet index catenoidal faces}
Let $\Sigma$ be a catenoidal end parameterized as in Proposition \ref{prop: kernel(J) of cat-faces}, with parameter $T$. Then its zero-junction-trace Dirichlet Morse index is
\[
\operatorname{Ind}_{J}^{0}(\Sigma)=
\begin{cases}
1,&T<0,\\
0,&T\ge0.
\end{cases}
\]
The zero-trace bilinear kernel is trivial when $T\ne0$ and is spanned by $\tanh t$ when $T=0$. Consequently, for the faces in the $Y$-noid family, every non-graphical catenoidal face has Dirichlet Morse index one and every graphical catenoidal face has Dirichlet Morse index zero.
\end{lem}
\begin{proof}
The zero-trace form decomposes into angular Fourier sectors by Proposition \ref{prop: second assumption steklov}. For $n=1$ and a radial profile $a$ with $a(0)=0$, integration by parts gives
\[
\mathfrak q_{1}(a,a)=\int_0^\infty
\left|a'+\tanh(t+T)a\right|^2dt\ge0.
\]
This identity is first obtained for compactly supported profiles. The single-face cutoff approximation from Proposition \ref{prop: function space for noncompact index} extends it to the full zero-trace variational space.
For $n\ge2$, the coefficient $n^2-2\operatorname{sech}^2(t+T)$ is strictly positive, so these sectors are positive. Hence every negative Dirichlet direction is radial.

The radial operator is the half-line Schrödinger operator with potential $-2\operatorname{sech}^2(t+T)$. Its negative spectrum is discrete. If $-\mu^2<0$ is a radial Dirichlet eigenvalue, the finite-energy solution at infinity is, up to a constant factor,
\[
a_\mu(t)=\bigl(\mu+\tanh(t+T)\bigr)e^{-\mu(t+T)},
\qquad \mu>0.
\]
The condition $a_\mu(0)=0$ is equivalent to
\[
\mu=-\tanh T.
\]
By the one-dimensional min--max principle, the index of the radial Dirichlet form on $H_0^1(0,\infty)$ equals the number of negative Dirichlet eigenvalues of this half-line operator. Compactly supported smooth profiles are dense in $H_0^1(0,\infty)$, and the single-face cutoff argument in Proposition \ref{prop: function space for noncompact index} identifies the same index with the index on the full zero-trace radial space $\mathcal X_{i,0}^0$. Hence there is exactly one negative radial direction when $T<0$ and none when $T\ge0$. At zero energy the admissible radial Jacobi field is $\tanh(t+T)$, which has zero junction trace exactly when $T=0$. For the parametrizations used here, $T<0$ is the non-graphical case and $T\ge0$ is the graphical case.
\end{proof}

For the two open subfamilies, all three catenoidal parameters $T_i$ are nonzero, so Corollary \ref{coro: general index decom} applies to every angular frequency. By Lemma \ref{lem: dirichlet index catenoidal faces}, the total zero-trace Dirichlet contribution is one for $0<\alpha<\pi/6$ and two for $\pi/6<\alpha<\pi/3$. It remains to evaluate the two-dimensional boundary forms $\mathfrak b_n$ from Corollary \ref{coro: general index decom}. The resonant surface $Y_{\pi/6}$ is treated separately at the end of the section.
\medskip

For the remainder of the family computation, let $\eta_i$ be the outward conormal to $\Sigma_i$ along $\Gamma$ and set
\[
\alpha_1=\alpha,\qquad
\alpha_2=\alpha+\frac{2\pi}{3},\qquad
\alpha_3=\alpha-\frac{2\pi}{3},
\]
where $\alpha_i=\frac{\pi}{2}-\angle(\eta_i,\mathbf e_3)$. Put
\[
\beta_i=\langle H_{\Gamma},\eta_i\rangle.
\]
The parametrizations give
\[
\cos\alpha_i=-\tanh T_i,
\qquad
c_1\cosh T_1=c_2\cosh T_2=c_3\cosh T_3=:c=R_\Gamma>0.
\]
Since $H_\Gamma=-c^{-1}e_\rho$ and $\langle\eta_i,e_\rho\rangle=\cos\alpha_i$, we obtain
\[
\beta_i=-\frac{\cos\alpha_i}{c}.
\]

\begin{lem}\label{lem: radial boundary sign}
Let $0<\alpha<\pi/3$, $\alpha\neq\pi/6$, and define $\alpha_i$ as above. Then the radial boundary form on
\[
V_\Delta=\{C\in\mathbb R^3:C_1+C_2+C_3=0\}
\]
is
\[
\mathfrak b_0(C,D)=\sum_{i=1}^3\frac{C_iD_i}{\cos\alpha_i}.
\]
Moreover,
\[
C^{(0)}=(\cos\alpha_1,\cos\alpha_2,\cos\alpha_3)
\]
spans $\ker\mathfrak b_0$. On the complementary line generated by $D=(0,1,-1)$,
\[
\mathfrak b_0(D,D)
=\frac1{\cos\alpha_2}+\frac1{\cos\alpha_3}
=-\frac{4\cos\alpha}{4\cos^2\alpha-3}.
\]
$\mathfrak b_0$ therefore has one negative direction and a one-dimensional kernel for $0<\alpha<\pi/6$, while it is positive semidefinite with a one-dimensional kernel for $\pi/6<\alpha<\pi/3$.
\end{lem}
\begin{proof}
For $n=0$, Proposition \ref{prop: kernel(J) of cat-faces} gives
\[
\delta_{i,0}-\beta_i=\frac{1}{c\cos\alpha_i}.
\]
Since $R_\Gamma=c$, the definition of the modewise boundary form gives
\[
\mathfrak b_0(C,D)
=R_\Gamma\sum_{i=1}^3(\delta_{i,0}-\beta_i)C_iD_i
=\sum_{i=1}^3\frac{C_iD_i}{\cos\alpha_i}.
\]
The identity
\[
\cos\alpha_1+\cos\alpha_2+\cos\alpha_3=0
\]
shows that $C^{(0)}\in V_\Delta$. For every $D\in V_\Delta$,
\[
\mathfrak b_0(C^{(0)},D)=\sum_{i=1}^3D_i=0,
\]
so $C^{(0)}\in\ker\mathfrak b_0$. Since $\cos\alpha_1>0$, the vectors $C^{(0)}$ and $(0,1,-1)$ are linearly independent and hence form a basis of $V_\Delta$.

Using $\alpha_2=\alpha+2\pi/3$ and $\alpha_3=\alpha-2\pi/3$, we compute
\[
\frac1{\cos\alpha_2}+\frac1{\cos\alpha_3}
=-\frac{4\cos\alpha}{4\cos^2\alpha-3}.
\]
The numerator is negative for $0<\alpha<\pi/3$. The denominator is positive for $0<\alpha<\pi/6$ and negative for $\pi/6<\alpha<\pi/3$. Thus the displayed value is negative in the first interval and positive in the second. In particular it is nonzero, so the kernel is exactly the line spanned by $C^{(0)}$, and the asserted signs follow.
\end{proof}

\subsection{The Subfamily \texorpdfstring{$0<\alpha<\pi/6$}{0 < alpha < pi/6}}

Let $0<\alpha<\pi/6$. Then $\Sigma_1$ is non-graphical and $\Sigma_2,\Sigma_3$ are graphical. Equivalently,
\[
\cos\alpha_1>0,\qquad \cos\alpha_2<0,\qquad \cos\alpha_3<0.
\]
Since $\cos\alpha_i=-\tanh T_i$, this gives
\[
T_1<0<T_2,T_3,
\]
so in particular all three catenoidal parameters are nonzero. Moreover,
\[
\beta_1<0<\beta_2,\beta_3.
\]

\begin{thm} \label{thm: index (0,pi/6)}
For $0<\alpha<\pi/6$, the surface $Y_\alpha$ has Morse index two and nullity five.
\end{thm}
\begin{proof}
By Lemma \ref{lem: dirichlet index catenoidal faces}, the non-graphical face $\Sigma_1$ has Dirichlet Morse index one and the graphical faces $\Sigma_2,\Sigma_3$ have Dirichlet Morse index zero. Hence
\[
\sum_{i=1}^3 \operatorname{Ind}_{J_i}^0(\Sigma_i)=1.
\]
Since every $T_i$ is nonzero, Corollary \ref{coro: general index decom} applies to every angular frequency.

For the radial sector $n=0$, Lemma \ref{lem: radial boundary sign} gives
\[
\operatorname{Ind}(\mathfrak b_0;V_\Delta)=1,
\qquad
\ker\mathfrak b_0
=\operatorname{span}\{C^{(0)}\},
\]
where
\[
C^{(0)}=(\cos\alpha_1,\cos\alpha_2,\cos\alpha_3).
\]
The identity
\[
\cos\alpha_1+\cos\alpha_2+\cos\alpha_3=0
\]
shows that $C^{(0)}\in V_\Delta$, and
\[
(\delta_{i,0}-\beta_i)C_i^{(0)}
=\frac{1}{c\cos\alpha_i}\cos\alpha_i
=\frac1c,
\qquad i=1,2,3.
\]
Thus this radial null direction satisfies the full coupled junction condition.

For every positive frequency, Proposition \ref{prop: kernel(J) of cat-faces} and the relations
$\cos\alpha_i=-\tanh T_i$ and $c_i\cosh T_i=c$ give
\[
\delta_{i,n}-\beta_i
=\frac{n^2-1}{c(n-\cos\alpha_i)},
\qquad n\ge1.
\]
When $n=1$, all three coefficients vanish. Hence
\[
\mathfrak b_1\equiv0\quad\text{on }V_\Delta.
\]
The cosine and sine sectors therefore contribute two weak-kernel dimensions each and no negative direction. When $n\ge2$, we have
\[
n-\cos\alpha_i\ge n-1>0,
\]
so every coefficient $\delta_{i,n}-\beta_i$ is positive. Thus $\mathfrak b_n$ is positive definite on $V_\Delta$ and has trivial kernel for every $n\ge2$.

Corollary \ref{coro: general index decom} now gives
\[
\operatorname{Ind}(Y_\alpha)
=\underbrace{1}_{\text{zero-trace Dirichlet part}}
+\underbrace{1}_{n=0\text{ boundary part}}
+2\sum_{n\ge1}\underbrace{\operatorname{Ind}(\mathfrak b_n;V_\Delta)}_{0}
=2.
\]
For the nullity, the bilinear kernels of the zero-trace face forms are trivial in every nonresonant sector by Corollary \ref{coro: general index decom}. Proposition \ref{prop: second assumption steklov} and the sector kernel formula therefore give
\[
\operatorname{nul}_{\mathcal D}(Y_\alpha)
=\underbrace{1}_{n=0}
+\underbrace{2}_{\cos\theta\text{ sector}}
+\underbrace{2}_{\sin\theta\text{ sector}}
=5.
\]
These five directions therefore lie in the full bilinear weak kernel.
\end{proof}

\subsection{The Subfamily \texorpdfstring{$\pi/6<\alpha<\pi/3$}{pi/6 < alpha < pi/3}}
Let $\pi/6<\alpha<\pi/3$. Then $\Sigma_1$ and $\Sigma_3$ are non-graphical, while $\Sigma_2$ is graphical. Hence $\eta_1,\eta_3$ point outside $\Gamma$ and $\eta_2$ points inside, so
\[
\beta_1,\beta_3<0<\beta_2.
\]
Equivalently,
\[
\cos\alpha_1>0,\qquad \cos\alpha_3>0,\qquad \cos\alpha_2<0.
\]
Since $\cos\alpha_i=-\tanh T_i$, this gives
\[
T_1,T_3<0<T_2,
\]
so in particular every catenoidal parameter $T_i$ is nonzero.

\begin{thm}
For $\pi/6<\alpha<\pi/3$, the surface $Y_\alpha$ has Morse index two and nullity five.
\end{thm}
\begin{proof}
By Lemma \ref{lem: dirichlet index catenoidal faces}, the two non-graphical faces have Dirichlet Morse index one and the graphical face has Dirichlet Morse index zero. Therefore
\[
\sum_{i=1}^3\operatorname{Ind}_{J_i}^0(\Sigma_i)=2.
\]
Since every $T_i$ is nonzero, Corollary \ref{coro: general index decom} applies to every angular frequency.

For the radial mode $n=0$,
\[
\delta_{i,0}-\beta_i=\frac{1}{c\cos\alpha_i}.
\]
Lemma \ref{lem: radial boundary sign} shows that the corresponding boundary form $\mathfrak b_0$ is positive semidefinite on $V_\Delta$ and has one-dimensional kernel
\[
\ker\mathfrak b_0
=\operatorname{span}\{C^{(0)}\},
\qquad
C^{(0)}=(\cos\alpha_1,\cos\alpha_2,\cos\alpha_3).
\]
The coupled kernel condition is also explicit:
\[
(\delta_{i,0}-\beta_i)C_i^{(0)}=\frac1c,
\qquad i=1,2,3.
\]
Thus the radial boundary form contributes no negative direction and exactly one weak-kernel direction.

For every positive frequency, Proposition \ref{prop: kernel(J) of cat-faces} gives
\[
\delta_{i,n}-\beta_i
=\frac{n^2-1}{c(n-\cos\alpha_i)},
\qquad n\ge1.
\]
For $n=1$ these coefficients vanish on all three faces, so
\[
\mathfrak b_1\equiv0\quad\text{on }V_\Delta.
\]
Hence the cosine sector contributes the two-dimensional plane $V_\Delta$ to the weak kernel, and the sine sector contributes another two dimensions. For $n\ge2$, since $|\cos\alpha_i|<1$,
\[
n-\cos\alpha_i>0
\]
for every $i$, and therefore $\delta_{i,n}-\beta_i>0$. Thus $\mathfrak b_n$ is positive definite on $V_\Delta$ and has trivial kernel.

Corollary \ref{coro: general index decom} also shows that the zero-junction-trace bilinear kernel is trivial in every nonresonant sector. There are no further weak-kernel directions. The index and nullity are therefore
\[
\operatorname{Ind}(Y_\alpha)
=\underbrace{2}_{\text{Dirichlet}}+
\underbrace{0}_{n=0\text{ boundary}}+
\underbrace{0}_{n\ge1\text{ boundary}}
=2,
\]
and
\[
\operatorname{nul}_{\mathcal D}(Y_\alpha)
=\underbrace{1}_{n=0}+
\underbrace{2}_{\cos\theta}+
\underbrace{2}_{\sin\theta}
=5.
\]
This is the full bilinear weak nullity, and the complete Fourier decomposition rules out additional null directions.
\end{proof}

\subsection{The Resonant Surface \texorpdfstring{$Y_{\pi/6}$}{Y pi/6}}
\begin{thm}
The surface $Y_{\pi/6}$ has Morse index two and nullity five.
\end{thm}

\begin{proof}
Let $M=(\Sigma_1,\Sigma_2,\Sigma_3;\Gamma)=Y_{\pi/6}$. Here $\Sigma_1$ is non-graphical, $\Sigma_2,\Sigma_3$ are graphical, and
\[
\alpha_1=\frac{\pi}{6},\qquad
\alpha_2=\frac{5\pi}{6},\qquad
\alpha_3=-\frac{\pi}{2}.
\]
Hence $T_1,T_2\neq0$ and $T_3=0$. The radial sector is therefore resonant only on $\Sigma_3$.

\medskip
\noindent Step 1. \emph{Positive angular frequencies.}
For a zero-junction-trace profile in frequency $n=1$, the catenoidal one-dimensional form satisfies
\[
\mathfrak q_{i,1}(a_i,a_i)
=\int_0^\infty\left|a_i'+\tanh(t+T_i)a_i\right|^2\,dt\geq0.
\]
Indeed, expanding the square and integrating the cross term gives the $n=1$ form. The boundary term at $t=0$ vanishes because $a_i(0)=0$, and the term at infinity vanishes by the cutoff approximation in Proposition~\ref{prop: function space for noncompact index}. For $n\ge2$,
\[
n^2-2\operatorname{sech}^2(t+T_i)\ge n^2-2>0,
\]
so every positive-frequency zero-trace form is nonnegative.

For the boundary part, Corollary~\ref{coro: general index decom} gives
\[
\delta_{i,1}-\beta_i=0,
\qquad
\delta_{i,n}-\beta_i
=\frac{n^2-1}{R_\Gamma(n-\cos\alpha_i)}>0,
\quad n\ge2.
\]
Thus $\mathfrak b_1\equiv0$ and $\mathfrak b_n$ is positive definite for $n\ge2$. Consequently
\[
\iota_n=0\qquad(n\ge1).
\]
Hence all negative directions lie in the radial sector, and
\[
\operatorname{Ind}(Q;\mathcal D(M))=\iota_0.
\]

\medskip
\noindent Step 2. \emph{The zero-trace radial space.}
Let $\mathcal X_0$ be the radial variational space and $\mathcal X_0^0$ its zero-junction-trace subspace, as introduced before Corollary~\ref{coro: general index decom}. By Lemma~\ref{lem: dirichlet index catenoidal faces} the Dirichlet indices of $\Sigma_1,\Sigma_2,\Sigma_3$ are $1,0,0$. Step~1 shows that the positive-frequency zero-trace forms are nonnegative, so the unique zero-trace negative direction on $\Sigma_1$ is radial. Therefore
\[
\operatorname{Ind}(\mathfrak q_0;\mathcal X_0^0)=1.
\]

The bilinear kernel of $\mathfrak q_0$ on $\mathcal X_0^0$ is one-dimensional. Indeed, Proposition~\ref{prop: kernel(J) of cat-faces} shows that an admissible radial Jacobi field on a catenoidal face is a multiple of $\tanh(t+T_i)$. This profile has nonzero junction value when $T_i\neq0$ and has zero junction value when $T_i=0$. Hence only $\Sigma_3$ contributes a zero-trace radial Jacobi field, and
\[
\ker\bigl(\mathfrak q_0|_{\mathcal X_0^0}\bigr)
=\operatorname{span}\{k\},
\qquad
k=(0,0,\tanh t).
\]

For $i=1,2$, let
\[
p_i(t)=\frac{\tanh(t+T_i)}{\tanh T_i}
\]
be the admissible radial Jacobi profile with junction value one. Set
\[
e=(p_1,-p_2,0),
\qquad
\mathcal Y=\{a\in\mathcal X_0:a_3^\Gamma=0\}.
\]
The trace of every $a\in\mathcal Y$ is of the form $(s,-s,0)$. Thus, with $s=a_1^\Gamma$,
\[
a=(a-se)+se,
\qquad a-se\in\mathcal X_0^0,
\]
and the decomposition is unique. Hence
\[
\mathcal Y=\mathcal X_0^0\oplus\operatorname{span}\{e\}
\]
algebraically.

Because $p_1$ and $p_2$ are Jacobi profiles, the boundary-pairing identity gives
\[
\mathfrak q_0(e,w)=0
\qquad(w\in\mathcal X_0^0).
\]
Moreover,
\[
\delta_{i,0}-\beta_i=\frac{1}{R_\Gamma\cos\alpha_i},
\qquad i=1,2.
\]
Since $\cos(\pi/6)=-\cos(5\pi/6)$,
\[
(\delta_{1,0}-\beta_1)+(\delta_{2,0}-\beta_2)=0,
\]
and therefore $\mathfrak q_0(e,e)=0$. We obtain the orthogonal decomposition
\[
\mathcal Y
=\mathcal X_0^0\oplus_{\mathfrak q_0}\operatorname{span}\{e\},
\]
so
\[
\operatorname{Ind}(\mathfrak q_0;\mathcal Y)=1,
\qquad
\ker(\mathfrak q_0|_{\mathcal Y})=\operatorname{span}\{k,e\}.
\]

\medskip
\noindent Step 3. \emph{The remaining radial trace direction and the Morse index.}
The space $\mathcal Y$ has codimension one in $\mathcal X_0$. To add the remaining trace direction, take
\[
h(t)=\operatorname{sech}t
\]
on $\Sigma_3$ and define
\[
\psi=(0,-p_2,h).
\]
Since $h(0)=1$, the junction trace of $\psi$ is $(0,-1,1)$, and therefore
\[
\mathcal X_0=\mathcal Y\oplus\operatorname{span}\{\psi\}
\]
algebraically.

The key pairing with the resonant Dirichlet Jacobi field is finite. Since $k(0)=0$, the junction term vanishes, and only the $\Sigma_3$ interior form remains. Using $k(t)=\tanh t$ and the equation $k''+2\operatorname{sech}^2t\,k=0$,
\[
\begin{aligned}
\mathfrak q_0(k,\psi)
&=\int_0^\infty
\left(k'h'-2\operatorname{sech}^2t\,kh\right)\,dt\\
&=[k'h]_{0}^{\infty}=-1.
\end{aligned}
\]

Because $\mathcal Y$ has codimension one in $\mathcal X_0$, adjoining one vector can raise the index by at most one. More explicitly, if $L\subset\mathcal X_0$ is negative definite, then $\dim(L\cap\mathcal Y)\ge\dim L-1$. Hence
\[
\iota_0\le
\operatorname{Ind}(\mathfrak q_0;\mathcal Y)+1=2.
\]
For the reverse inequality, choose $v\in\mathcal Y$ with $\mathfrak q_0(v,v)<0$ and set
\[
\psi'
=\psi-
\frac{\mathfrak q_0(v,\psi)}{\mathfrak q_0(v,v)}v.
\]
Then
\[
\mathfrak q_0(v,\psi')=0,
\qquad
\mathfrak q_0(k,\psi')=-1,
\]
because $k$ lies in the bilinear kernel of $\mathfrak q_0|_{\mathcal Y}$. On $\operatorname{span}\{k,\psi'\}$ the matrix of the form is
\[
\begin{pmatrix}
0&-1\\
-1&\mathfrak q_0(\psi',\psi')
\end{pmatrix},
\]
whose determinant is $-1$. Hence this two-dimensional restriction has one positive and one negative direction. Its negative direction is $\mathfrak q_0$-orthogonal to the negative line $\operatorname{span}\{v\}$. These two negative directions span a negative-definite two-dimensional subspace of $\mathcal X_0$. Therefore
\[
\iota_0\ge2.
\]
Combining the two bounds gives
\[
\iota_0=2.
\]
Together with Step~1,
\[
\operatorname{Ind}(Y_{\pi/6})=2.
\]

\medskip
\noindent Step 4. \emph{The weak nullity.}
Let $u=y+s\psi\in\ker\mathfrak q_0$, with $y\in\mathcal Y$. Testing against $k$ gives
\[
0=\mathfrak q_0(u,k)
=s\,\mathfrak q_0(\psi,k)=-s,
\]
so $s=0$. Hence $u=y\in\ker(\mathfrak q_0|_{\mathcal Y})=\operatorname{span}\{k,e\}$. It remains to impose the equation against the one missing test direction $\psi$. The corresponding functional is nonzero because
\[
\mathfrak q_0(k,\psi)=-1.
\]
Therefore it cuts the two-dimensional space $\operatorname{span}\{k,e\}$ down to a one-dimensional radial weak kernel.

The radial kernel line can be written explicitly. The boundary-pairing identity gives
\[
\mathfrak q_0(e,\psi)
=R_\Gamma(\delta_{2,0}-\beta_2)
=\frac{1}{\cos\alpha_2}
=-\frac{2}{\sqrt3}.
\]
Thus
\[
\ker\mathfrak q_0
=\operatorname{span}\left\{e-\frac{2}{\sqrt3}k\right\}.
\]

For $n=1$, Corollary~\ref{coro: general index decom} gives a trivial zero-trace bilinear kernel and $\mathfrak b_1\equiv0$ on the two-dimensional compatible trace plane. Hence each of the cosine and sine sectors contributes two weak-kernel dimensions. For $n\ge2$, both the zero-trace kernel and the boundary kernel are trivial. Proposition~\ref{prop: second assumption steklov} therefore gives
\[
\operatorname{nul}_{\mathcal D}(Y_{\pi/6})
=\underbrace{1}_{n=0}
+\underbrace{2}_{\cos\theta}
+\underbrace{2}_{\sin\theta}
=5.
\]
\end{proof}

\bibliographystyle{amsalpha}
%\bibliography{main}

\providecommand{\bysame}{\leavevmode\hbox to3em{\hrulefill}\thinspace}
\providecommand{\MR}{\relax\ifhmode\unskip\space\fi MR }
% \MRhref is called by the amsart/book/proc definition of \MR.
\providecommand{\MRhref}[2]{%
  \href{http://www.ams.org/mathscinet-getitem?mr=#1}{#2}
}
\providecommand{\href}[2]{#2}

\end{document}